\documentclass[10pt]{amsart}
\usepackage{geometry,amsmath,amsthm,amssymb}                
\usepackage[parfill]{parskip}    
\usepackage{graphicx}
\usepackage{multicol}
\usepackage{epstopdf}
\usepackage{mathtools}
\usepackage{flafter}
\usepackage{float}
\usepackage{caption}
\usepackage{subcaption}
\usepackage{wrapfig}

\input xy 
\xyoption{all}

\floatstyle{boxed} 
\restylefloat{figure}

\newcommand{\C}{\mathbb{C}}
\newcommand{\R}{\mathbb{R}}

\theoremstyle{definition}

\theoremstyle{plain}
\newtheorem{thm}{Theorem}[section]

\theoremstyle{remark}

\numberwithin{equation}{section}

\begin{document}

\title{The Real Dynamics of Bieberbach's Example}

\author{S. Hayes, A. Hundemer, E. Milliken, T. Moulinos*}

\maketitle

 ABSTRACT.\small { Bieberbach constructed in 1933 domains in $ \bf {C}^2$ which were biholomorphic to $ \bf {C}^2$ but not dense. The existence of such domains was unexpected. The special domains Bieberbach considered are basins of attraction of a cubic H\'enon map. This classical method of construction is one of the first applications of dynamical systems to complex analysis. In this paper, the  boundaries of the real sections of Bieberbach's domains will be calculated explicitly as the stable manifolds of the saddle points. The real filled Julia sets and the real Julia sets of Bieberbach's map will also be calculated explicitly and illustrated with computer generated graphics. Basic differences between real and the complex dynamics will be shown.}

 \section{INTRODUCTION}
\normalsize{Bieberbach constructed in 1933 domains in $\C^2$ biholomorphic to $\C^2$ but omitting an open set. The existence of these domains was unexpected, because the analogous statement for the one-dimensional plane $\C$  is false due to Picard's theorem which insures that $\C$ is the only domain in the plane biholomorphic to $\C$.  Such domains $\Omega$ in $\C^2$ are referred to as \emph{Fatou-Bieberbach domains}. Their classical method of construction, due to Fatou, is one of the first applications of dynamical systems to complex analysis.}

Bieberbach considered two domains $\Omega^+_\C$ and $\Omega^-_\C$  which are basins  of attraction of the same automorphism
$$f(z,w):=(w, \frac{z}{2}-w^3+\frac{3}{4}w)$$
Both basins are biholomorphic to $\C^2$, according to a result originating with Poincar\'{e}, but obviously not all of  $\C^2$, since they are disjoint.  These basins are symmetric with respect to the origin, and Bieberbach's map $f$ is one of the simplest having two basins in such a geometric relationship to each other.

For over two decades now, with the incentive from the visualization possibilities offered by computer graphics, renewed interest in higher dimensional complex dynamics has led to many interesting topological results. For example, the basins $\Omega^+_\C$ and $\Omega^-_\C$ have the same boundary in $\C^2$ and that boundary is never a topological manifold ([BS2], Theorem 2).  Surprisingly enough, however,  computer pictures of the real sections of these boundaries look smooth.  The purpose of this paper is to present a proof of that fact. It will be shown that the boundaries of  the real basins $\Omega_+:=\Omega^+_\C\cap\R^2$  and $\Omega_-:=\Omega^-_\C\cap\R^2$ in $\R^2$ coincide and are composed exactly of the real stable manifolds of 3 saddle points. Whereas in the standard literature it is sometimes stated that basin boundaries are smooth on the basis of computer studies or numerical calculations (see [R], pg.503), an explicit proof is given here.

\scriptsize{2000 Mathematics Review Classification: 37C05, 37C25, 37E30\\
Key words: H\'enon maps. Basin boundary. Stable manifolds. Julia sets.}\\

\normalsize{There are several features of Bieberbach's domains illustrating basic differences between real and complex dynamics. Bieberbach's map leads to domains in $\R^2$ bianalytic to all of $\R^2$ whose boundaries coincide. However, in contrast to the complex case, they are not described as the closure of the real stable manifold of an arbitrary saddle point  ([BS2], Theorem 1).  Furthermore, in the complex case there are always infinitely many periodic points [FM], but in Bieberbach's example there are only 5  ([K], Proposition 6.5). In the complex case, the intersection of $\Omega^+_\C$ (resp. $\Omega^-_\C$) with a complex line is always bounded ([BS], Theorem 1), see also ( [K], Theorem 4.3]), whereas $\Omega_+$ and $\Omega_-$ are unbounded.
Another difference is that the boundary of the complex basin $\Omega^+_\C$ is also the boundary of all points in $\C^2$ with unbounded forward orbits [BS2], whereas the the boundary of the real basin 
$\Omega_+$ is not the boundary of the forward escaping set.}\\

The paper is organized as follows: In the next section, first a closed polygon $R$ in $\bf {R}^2$ will be shown to contain all real points with bounded forward as well as backward orbits. Then the fate of the forward and the backward orbit of every point in $R$ will be described. In the third section, the stable and the unstable manifolds of every saddle point will be located. The real filled Julia sets and the real Julia sets are calculated explicitly in the fourth section in terms of the 5 real periodic points. The real basin boundaries can be completely described in the last section which also contains a revealing computer generated image of those basins.\\

* We thank Jeff Galas, who generated Figures 1-7, and Korrigan Clark for their contributions.
\maketitle
\section{ORBIT BEHAVIOR}

Consider $f (x, y) = (y, \frac{x}{2} - y^3 + \frac{3}{4} y)$ as a self-map of $\bf {R}^2$. The points $p_+ = (\frac{1}{2},\frac{1}{2}),  p_- = (-\frac{1}{2}, -\frac{1}{2})$ are obviously fixed points. The third fixed point is at the origin  and that is a saddle. There is also a period two saddle at $p = (-\frac{\sqrt{5}}{2}, \frac{\sqrt{5}}{2})$, $p' = (\frac{\sqrt{5}}{2}, -\frac{\sqrt{5}}{2}).$ The backward iterate of $(x, y)$ is  $f^{-1}(x, y) = (2x^3 -\frac{3}{2} x + 2y, x)$. \\

\begin{wrapfigure}{Hr}{0.4\textwidth}
\begin{center}
\includegraphics[width=0.38\textwidth]{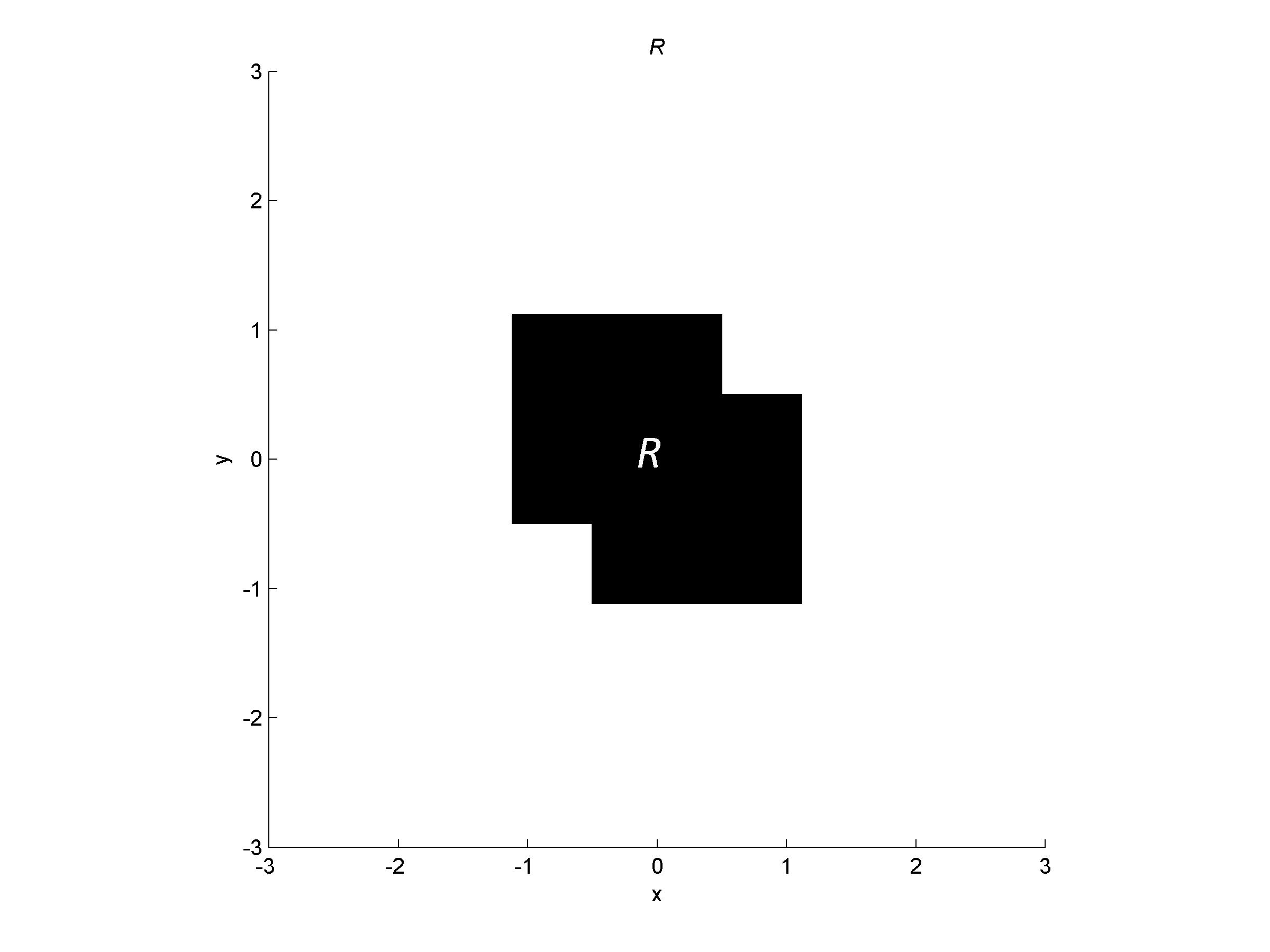}
\end{center}	
\caption{R}
\end{wrapfigure}
The first objective is to locate the set $K_\R$ of points in $\bf{R}^2$ with bounded forward and bounded backward orbits, since they are the observables. $K_\R$ also generates the set $K^+_\R$ of points with bounded forward orbits as well as the set $K^-_\R$ of points with bounded backward orbits (see section 4). We will use a partitioning of the real plane similar to that in [K, p.132]. As a first estimate it will be shown that $K_\R$ is contained in a closed polygon $R$ with corners given by the 8 points: $p, (\frac{1}{2}, \frac{\sqrt{5}}{2}), p_+, (\frac{\sqrt{5}}{2}, \frac{1}{2}), p', (-\frac{1}{2}, -\frac{\sqrt{5}}{2}), p_-, (-\frac{\sqrt{5}}{2}, -\frac{1}{2})$.\\ Moreover, except for the periodic points $ p, p', p_+, p_-$, the set $K_\R$ is in the interior of $R$. The proof  will use the backward iterates

\textbf{Lemma 2.1} $K_\R   \subset R $ and \\
$K_\R \setminus \{p, p', p_+, p_-\} \subset  int  \, R$

\begin{proof}
The proof will show that  outside of the interior of $R$ every point except $p, p', p_+, p_-$ escapes to infinity either under forward or under backward iteration of $f$.  \\
The complement of the interior of $R$  will be partitioned into the following $4$ closed quadrants $Q_k, 1 \leq k \leq 4$ and their reflections $Q_k ' = \sigma(Q_k)$ at the origin where $\sigma(x,y)= (-x, -y)$: 

\[
Q_1 : = \{(x,y) \in \R^2 : x \leq -1/2,  \hspace{1mm} y \leq -1/2 \},\qquad  Q_2 := \{(x,y) \in \R^2 : x \leq -\sqrt{5}/2,  \hspace{1mm} y \leq \sqrt{5}/2 \} 
\]
\[
Q_3 := \{(x,y) \in \R^2 : x \leq -\sqrt{5}/2,  \hspace{1mm} y \geq \sqrt{5}/2 \} , \qquad Q_4 := \{(x,y) \in \R^2 : x \geq -\sqrt{5}/2,  \hspace{1mm} y \geq \sqrt{5}/2 \} 
\]

\begin{figure}[H]
\begin{subfigure}[h]{0.28\textwidth}
\centering
\includegraphics[trim=162mm 0mm 10cm 0mm,clip=true,keepaspectratio=true,scale=0.075]{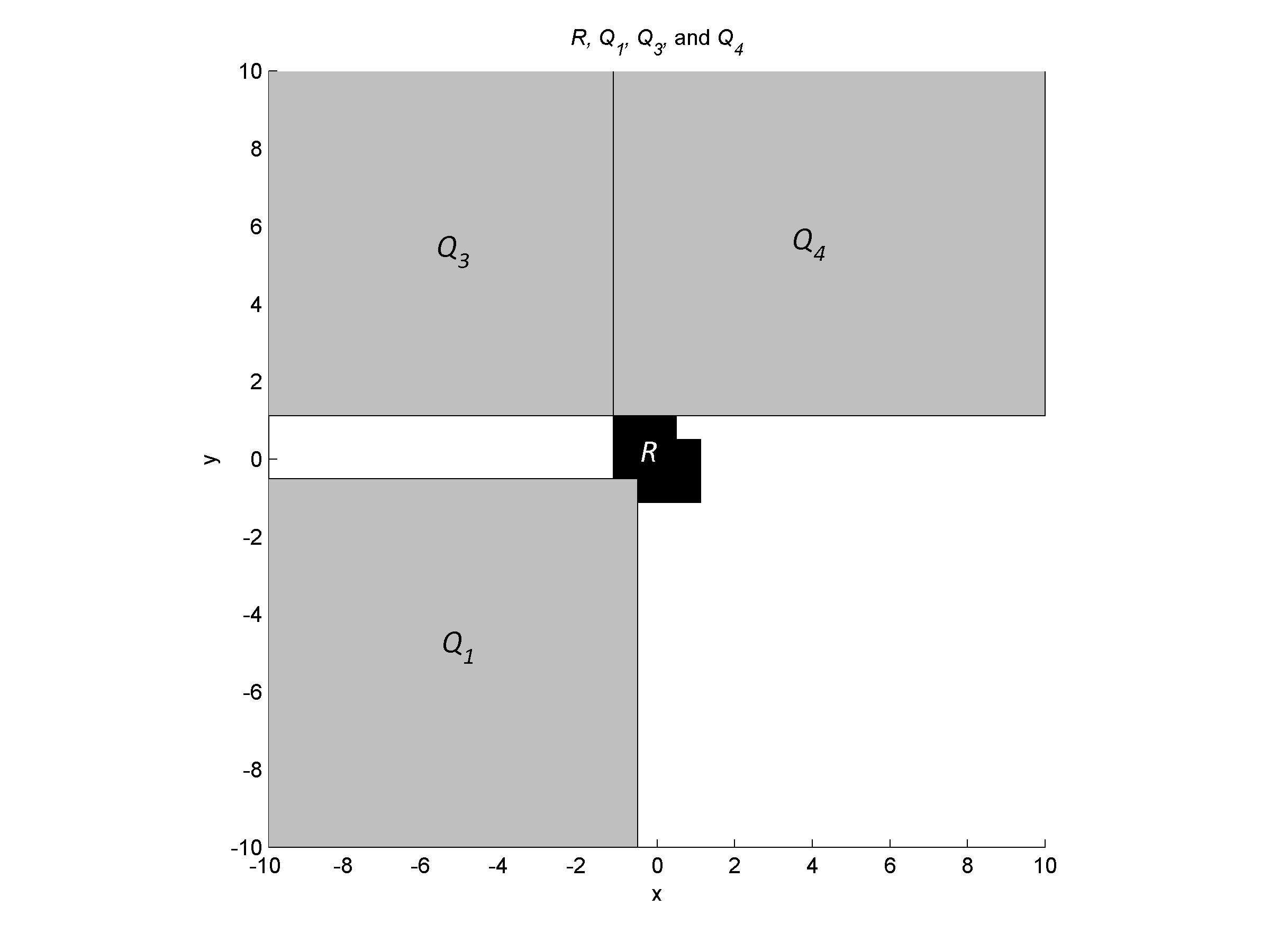}
\caption{R, $Q_1$, $Q_3$, $Q_4$}
\end{subfigure}
\quad
\begin{subfigure}[h]{0.28\textwidth}
\centering
\includegraphics[trim=16cm 0mm 0mm 0mm,clip=true,keepaspectratio=true,scale=0.075]{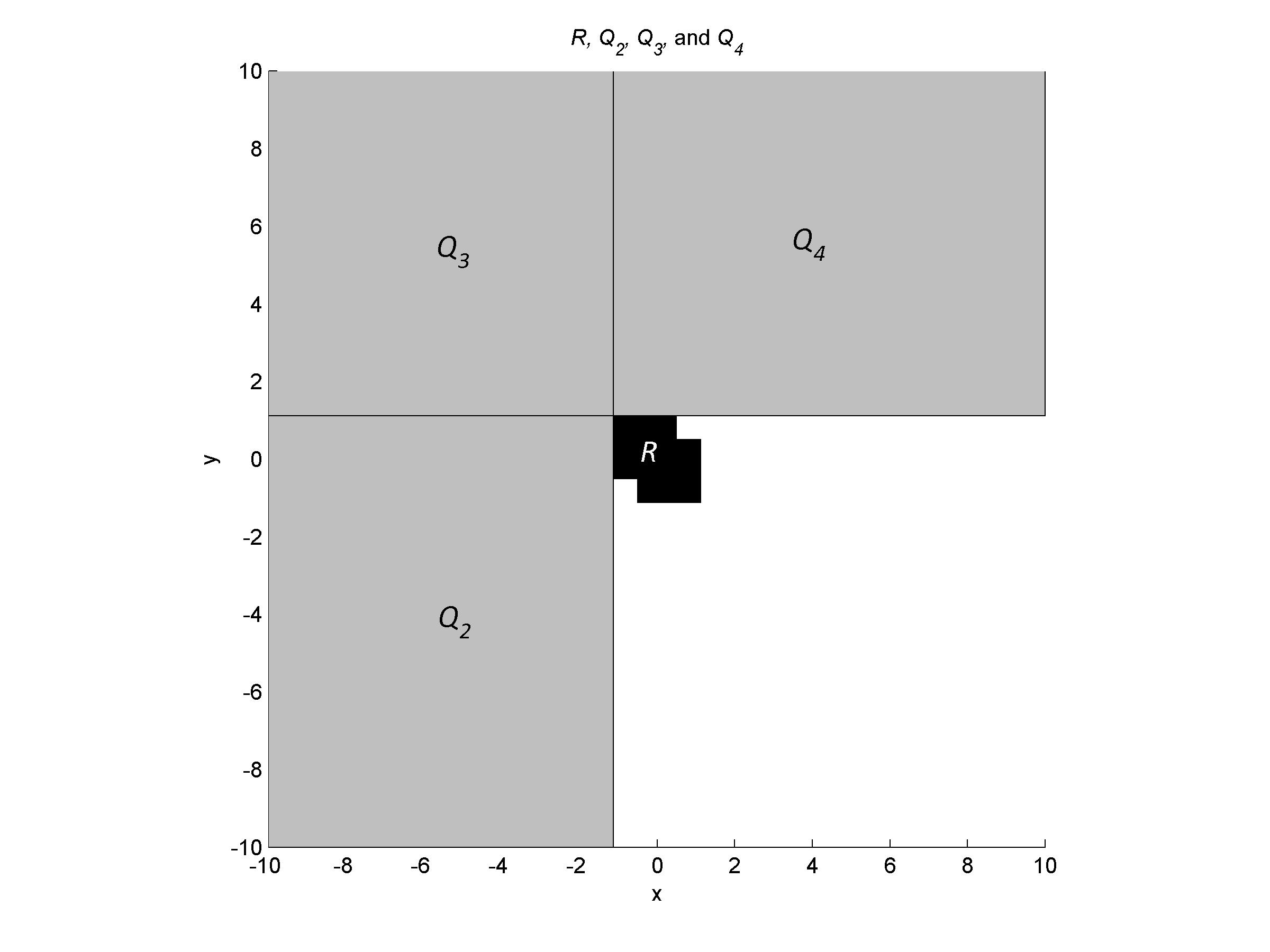}
\caption{R, $Q_2$, $Q_3$, $Q_4$}
\end{subfigure}
\begin{subfigure}[h!]{0.28\textwidth}
\centering
\includegraphics[trim=10cm 0cm 10cm 0mm,clip=true,keepaspectratio=true,scale=0.075]{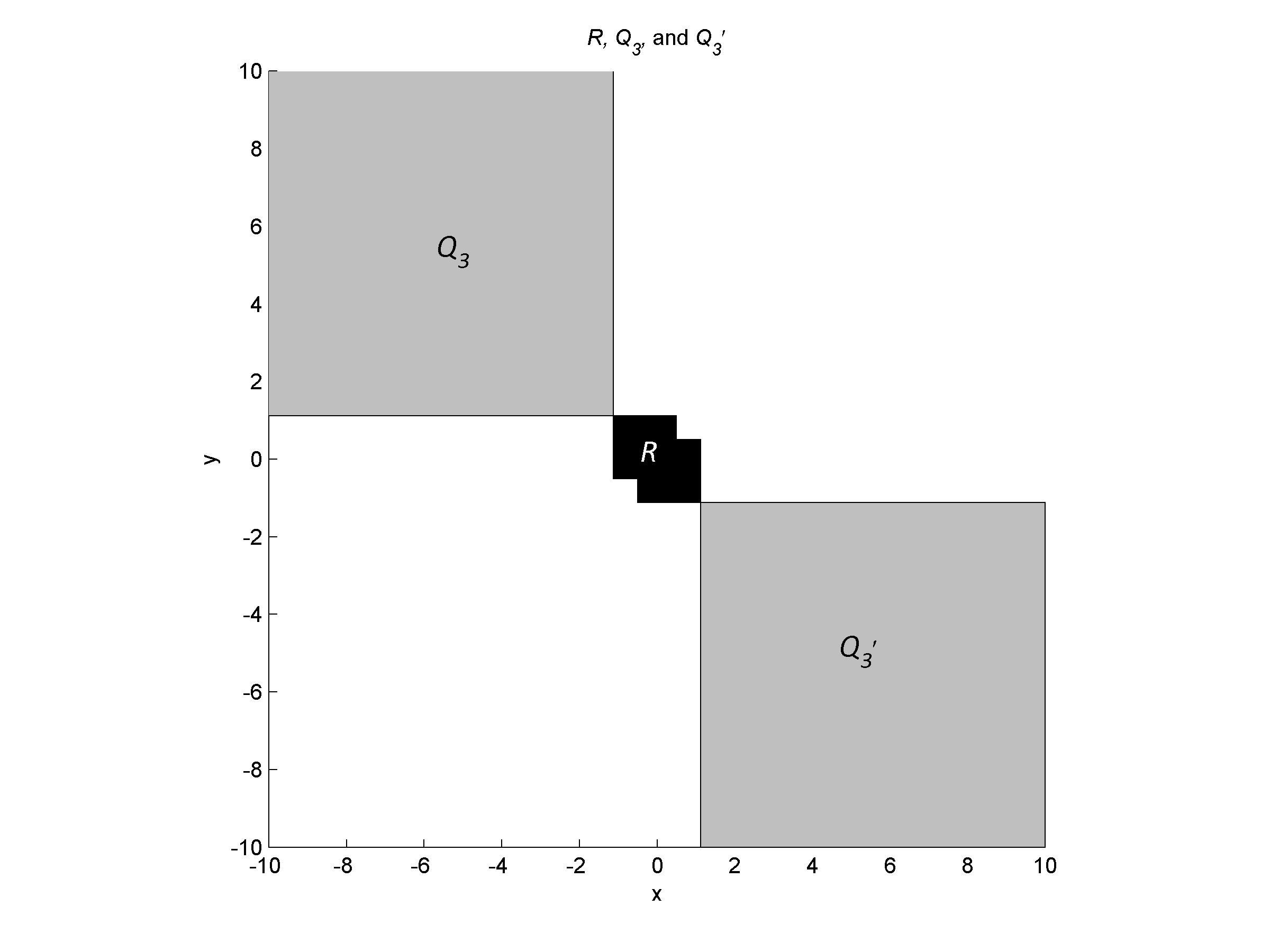}
\caption{R, $Q_3$, $Q'_3$}
\end{subfigure}
\caption{Partition of $\R$}
\end{figure}

It will be shown that  $f^n(x,y) \to \infty$ if $(x,y)$ is in $Q_3\setminus\{ p\}$ and $f^{-n}(x,y) \to \infty$ when $(x,y)$ is in $Q_1, Q_2,$ or $Q_4$ but is not $p$ or $ p_-$.  If we interchange $Q_k$ and $Q'_k$, the corresponding statements hold, since $f \circ \sigma= \sigma \circ f$ and $f^{-1} \circ \sigma = \sigma \circ f^{-1}$.  
The quadrants are mapped as follows: \[
f^{-1}(Q_1) \subset Q_1, f^{-1}(Q_2) \subset Q_4 ', f(Q_3) \subset Q_3',  f^{-1}(Q_4) \subset Q_2'
\]

\begin{figure}[H]
\begin{subfigure}[h]{0.28\textwidth}
\centering
\includegraphics[trim=16cm 0mm 5cm 0mm,clip=true,keepaspectratio=true,scale=0.075]{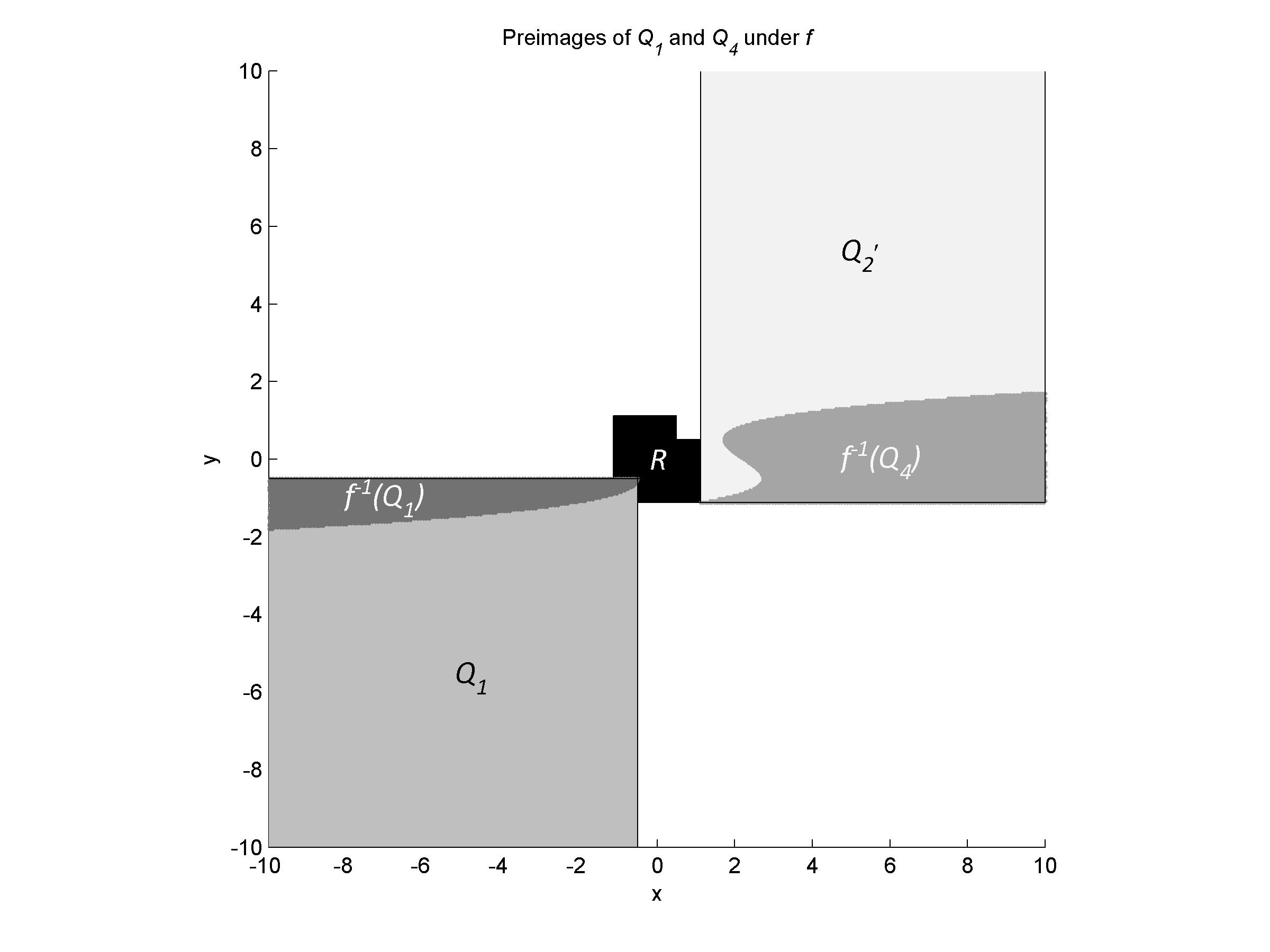}
\caption{$f^{-1}(Q_1)$ and $f^{-1}(Q_4)$}
\end{subfigure}
\quad
\begin{subfigure}[h]{0.28\textwidth}
\centering
\includegraphics[trim=13cm 0mm 0mm 0mm,clip=true,keepaspectratio=true,scale=0.075]{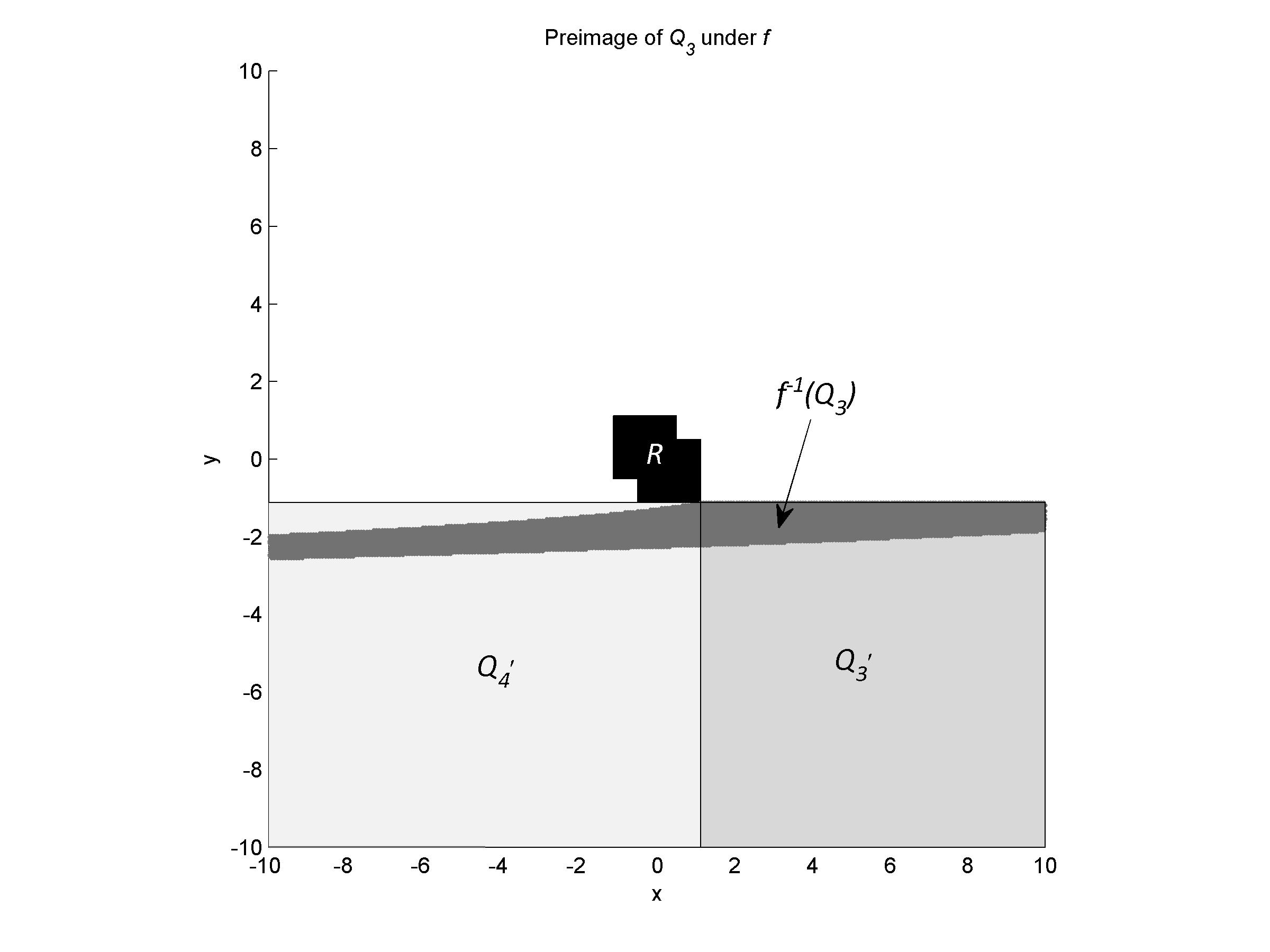}
\caption{$f^{-1}(Q_3)$}
\end{subfigure}
\begin{subfigure}[h!]{0.28\textwidth}
\centering
\includegraphics[trim=10cm 0cm 2mm 0mm,clip=true,keepaspectratio=true,scale=0.075]{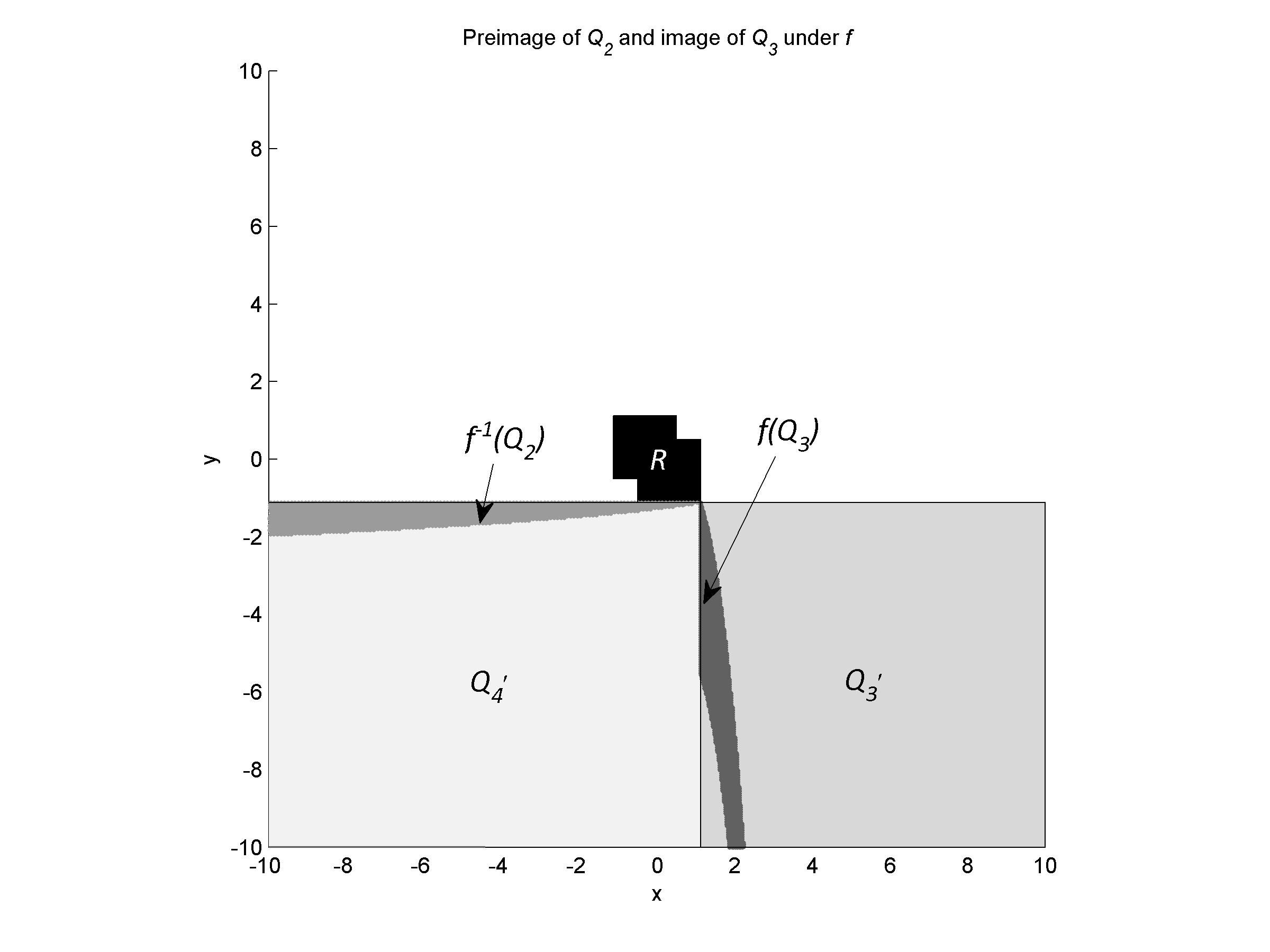}
\caption{$f^{-1}(Q_2)$ and $f(Q_3)$}
\end{subfigure}
\caption{Mappings of the partitions of $\R$}
\end{figure}

We will first consider $Q_3$. If $(x,y)$ lies in $Q_3$ but $y \not= \frac{\sqrt{5}}{2}$, then its image under $f$ lies further away from the origin with respect to the pseudonorm $|(x,y)|= |y - \frac{x}{2}|$ which is obviously the $y-$ intercept of the line through $(x, y)$ with slope $\frac{1}{2}$. To see this, notice that  $|(x, y)| = y - \frac{x}{2}$ if $(x, y) \in Q_3$ and when $(x, y) \in Q'_3$, then $|(x, y)| = \frac{x}{2} - y$. Since $f(Q_3) \subset Q'_3$, if $(x, y) \in Q_3$ with  $y \not= \frac{\sqrt{5}}{2}$, then

\[
|f(x,y)| - |(x,y)| = -\frac{x}{2} + y^3 -\frac{y}{4} - y + \frac{x}{2} = y(y^2 - \frac{5}{4}) >0.
\] 

Using the fact that the pseudonorm is  preserved by the reflection $\sigma$ at the origin, it immediately follows that when $(x, y)  \in Q'_3\setminus\{p'\}$, then $|f(x, y)| - |(x, y)| > 0$ if $y \neq -\frac{\sqrt{5}}{2}$. 

If $ (x, y) \in Q_3\setminus \{p\}$ with $y \neq \frac{\sqrt{5}}{2}$, then $y_1 \neq -\frac{\sqrt{5}}{2}$ for $f(x, y) = (x_1, y_1) \in Q'_3\setminus \{p'\}$  and  
\[
|f^2(x, y)| > |f(x, y)| > |(x, y)|.
 \]
The same inequalities hold for $(x, y) \in Q'_3 \setminus \{p'\}$ with $y \neq -\frac{\sqrt{5}}{2}$ due to the properties of $\sigma$.
By induction, for a point $(x,y)$ in $( Q_3 \cup Q'_3) \setminus \{p, p'\}$ with $|y| \neq \frac{\sqrt{5}}{2}$  the sequence $(|f^n (x,y)|)_{n \in \mathbb{N}}$ is strictly  monotonically increasing.\\
Whenever $| y| = \frac{\sqrt {5}}{2}$, after one iterate $|y_1| \neq \frac{\sqrt{5}}{2}$ for $f(x, y) = (x_1, y_1)$.  Consequently,  $(|f^n (x,y)|)_{n \geq 1}$ is  strictly  monotonically increasing for all points $(x,y)$ in $( Q_3 \cup Q'_3) \setminus \{p, p'\}$. That sequence is also unbounded; otherwise, it would converge to some value r where $|a|=r$ for every accumulation point $a$ of $(f^n(q))$.  Because $Q_3$ is closed, $a \in Q_3$. This is a contradiction, since then $|f(a)| > |a|$ would follow, contradicting the fact that $f(a)$ is also an accumulation point in $Q_3$ and therefore  $|f(a)| =|a|$.  It follows that $f^n(x,y) \to \infty$ as $n \to \infty$ for $(x,y) \in Q_3  $.

Now consider $Q_1$. If  $(x,y) \in Q_{1}$, then $f^{-1}(x,y) \in Q_{1}$.
Using the pseudonorm $\lvert (x,y) \rvert = \lvert y+\frac{x}{2} \rvert$ in $Q_1$,   we see that for $(x,y) \in Q_{1}$,

\[
\lvert (x,y) \rvert  = -y -  \frac{x}{2} \quad and  \quad \lvert f^{-1}(x,y) \rvert =\lvert (x_{-1},y_{-1}) \rvert = -y_{-1} - \frac{ x_{-1}}{2} = -x - x^{3} + \frac{3x}{4} - y = - \frac{x}{4} - x^3 - y.
\]

Therefore,
\[
\lvert f^{-1}(x,y) \rvert - \lvert (x,y) \rvert = -x - x^{3} + \frac{3x}{4} - y +y + \frac{x}{2} = -x^{3} + \frac{x}{4}  = -x (x^2 - \frac{1}{4}) .
\]

This difference is positive for $ x < -\frac{1}{2}$.  If  $x =  -\frac{1}{2}$ and $(x, y)\in Q_1$, then $f^{-1}(x,y) = (x_{-1}, y_{-1})$ satisfies $x_{-1} < - 1/2$ if and only if $y < -\frac{1}{2}$, implying by induction that for $(x, y) \in Q_1 \setminus p_-$ the sequence $(|f^{-n}(x,y)|)_{n \geq 1}$ is strictly monotonically increasing.  As above, it is unbounded and   $f^{-n}(x,y)\rightarrow \infty$ as $n \rightarrow \infty$ for $(x,y) \in Q_{1}\setminus p_-$ follows.

Finally, we will consider $Q_2$ and $Q_4$. Note that $Q_2$ and $Q_1$ overlap, and once any backward iterate of a point $(x, y)$ in $Q_2$ lands in $Q_1$, then its fate is sealed and  $f^{-n}(x,y)\rightarrow \infty$ as $n \rightarrow \infty$.  Thus, we only need to consider points  $(x,y) \in Q_{2} $ such that $f^{-n}(x, y) \notin Q_1$ for every $n$.

Using the maximum norm $\lvert (x,y) \rvert = max\{  \lvert x \rvert, \lvert y \rvert \}$, we arrive at the following:
If $ (x,y) \in Q_{2} \setminus Q_{1}$, then $x\leq -\frac{\sqrt{5}}{2},   -\frac{1}{2} \leq y \leq \frac{\sqrt{5}}{2}$ and $ \lvert (x,y) \rvert =\lvert x \rvert = - x.$  If  $(x, y) \in Q'_4 \setminus Q_1$, then  $ -\frac{1}{2} \leq x \leq \frac{\sqrt{5}}{2} ,  y \leq -\frac{\sqrt{5}}{2}$ and $ \lvert (x,y) \rvert =\lvert y \rvert = - y.$ Therefore, when  $(x, y) \in Q_2 \setminus Q_1$  and $f^{-1}(x, y) \in Q'_4 \setminus Q_1$, the difference\\
\[
\lvert f^{-1}(x,y) \rvert - \lvert (x,y) \rvert =-x +x = 0.
\]

However, when $(x, y) \in Q'_4 \setminus Q_1$, then it is no restriction to assume that   $f^{-1}(x, y) \in Q_2 \setminus Q_1$,  and  the difference \\
\[
\lvert f^{-1}(x,y) \rvert - \lvert (x,y) \rvert = -2x^3 +\frac{3}{2}x - y
\]
is positive if $ y <  -\frac{\sqrt{5}}{2}$.

Consequently, $ \lvert f^{-2(n+1)}(x,y) \rvert > \lvert f^{-2n}(x,y) \rvert$  and  $|f^{-(2n+1)}(x,y)| > | f^{-(2n-1)}(x,y)|$ for $n \in \mathbb{N}$ and $(x, y) \in Q_2 \setminus Q_1$ with  $ y <  -\frac{\sqrt{5}}{2}$. That means the sequence of the absolute values of the backward images is itself not strictly monotonically increasing, but the subsequence of the absolute values of the odd inverse images as well as  the subsequence of the absolute values of the even inverse images are both strictly monotonically increasing and therefore unbounded, implying that $f^{-n} (x, y) \to \infty $.

\end {proof}

 \textbf{Corollary 2.2}  The backward orbit of every point $q$ on the boundary of $R$ which is not $p, p', p_+$  or  $p_-$  escapes, i.e. $f^{-n}(q) \to \infty$.

\begin{proof}
Since $q$ lies in one of $Q_1, Q_2, Q_4  $ or their reflections, Lemma 2.1 gives the result.
\end{proof}

The orbit behavior inside $R$ will be studied in two steps. All points outside the  open unit square $S = \{(x, y)\in \R^2 : | (x, y) | < \frac{1}{2}\} $ will be denoted by $T$ and will be considered first;  $| (x, y)|$ denotes the maximum norm for points $(x, y)$ in $T$. 

 We subdivide $T$ as follows.   Let 
\[T_{1}= \{(x,y) \in T : -y \leq x \leq \frac{1}{2}, \frac{1}{2}  \leq y   \leq  \frac{\sqrt{5}}{2}\}, \qquad T_{2}= \{(x,y) \in T :  -\frac{\sqrt{5}}{2}  \leq x \leq -y,  \frac{1}{2}  \leq  y   <  -x \}
\]
\[
T_{3}= \{(x,y) \in T :\frac{-\sqrt{5}}{2}  \leq x \leq  \frac{-1}{2}, \frac{-1}{2}  \leq  y   \leq  \frac{1}{2}  \}
\]\\
As before,  $\sigma (x, y) = (-x, -y)$  denotes the reflection at the origin, and  for each $i \in \{1, 2, 3\}$, $\sigma(T_i)= T_i '$.  It is clear that $T_{+} \cup T_{-} = T$ where $T_{+} = \bigcup_{i=1} ^{n=3}T_{i}$ and $T_{-} = \bigcup_{i=1} ^{n=3}T_{i}'$. \\

\begin{figure}[H]
\begin{subfigure}[H]{0.28\textwidth}
\centering
\includegraphics[trim=16cm 0mm 0mm 0mm,keepaspectratio=true,scale=0.075]{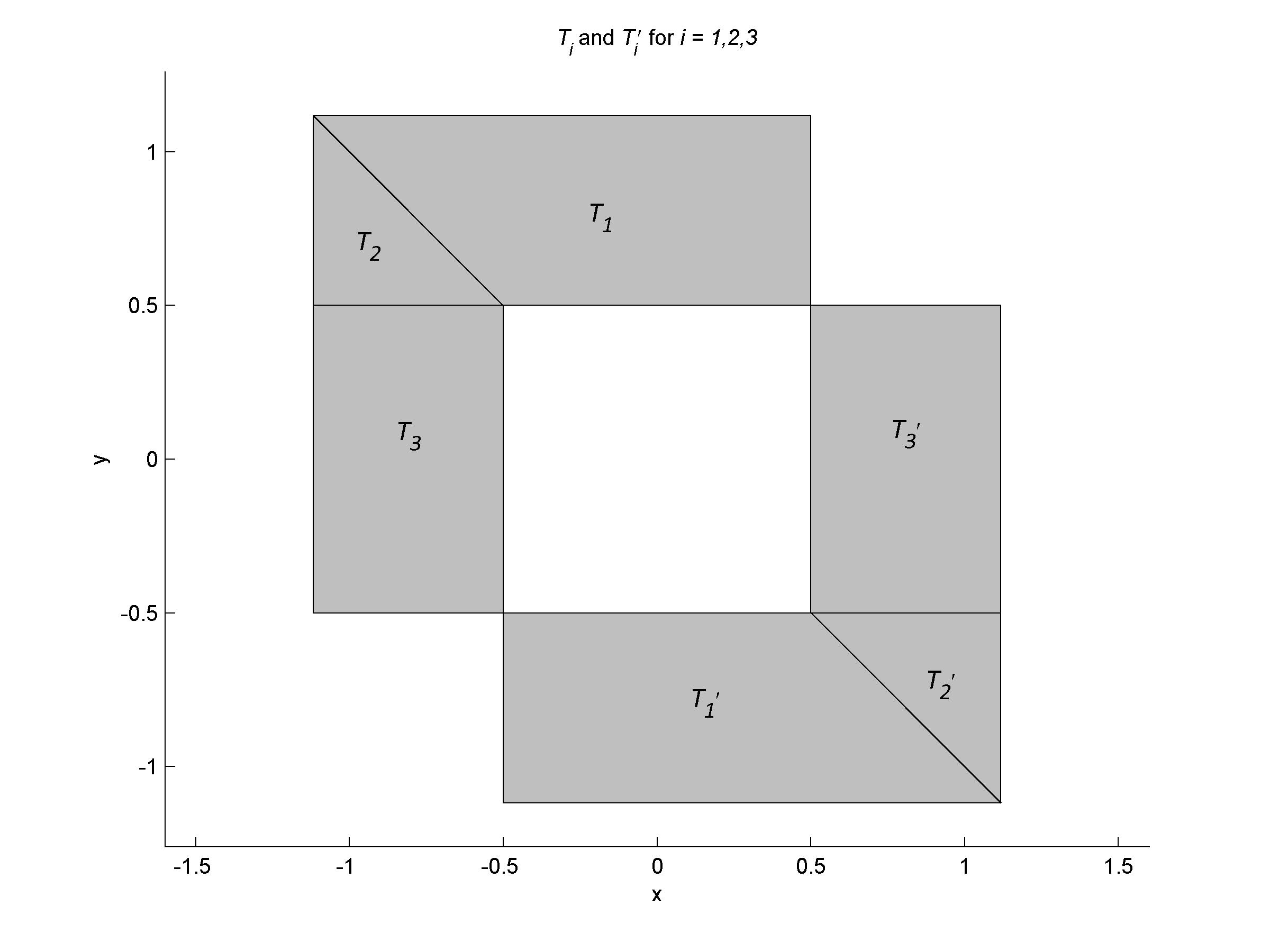}
\caption{$T_i$ and $T'_i$ for $i=1,2,3$}
\end{subfigure}
\quad\quad
\begin{subfigure}[H]{0.28\textwidth}
\centering
\includegraphics[trim=10cm 0mm 0mm 0mm,clip=true,keepaspectratio=true,scale=0.075]{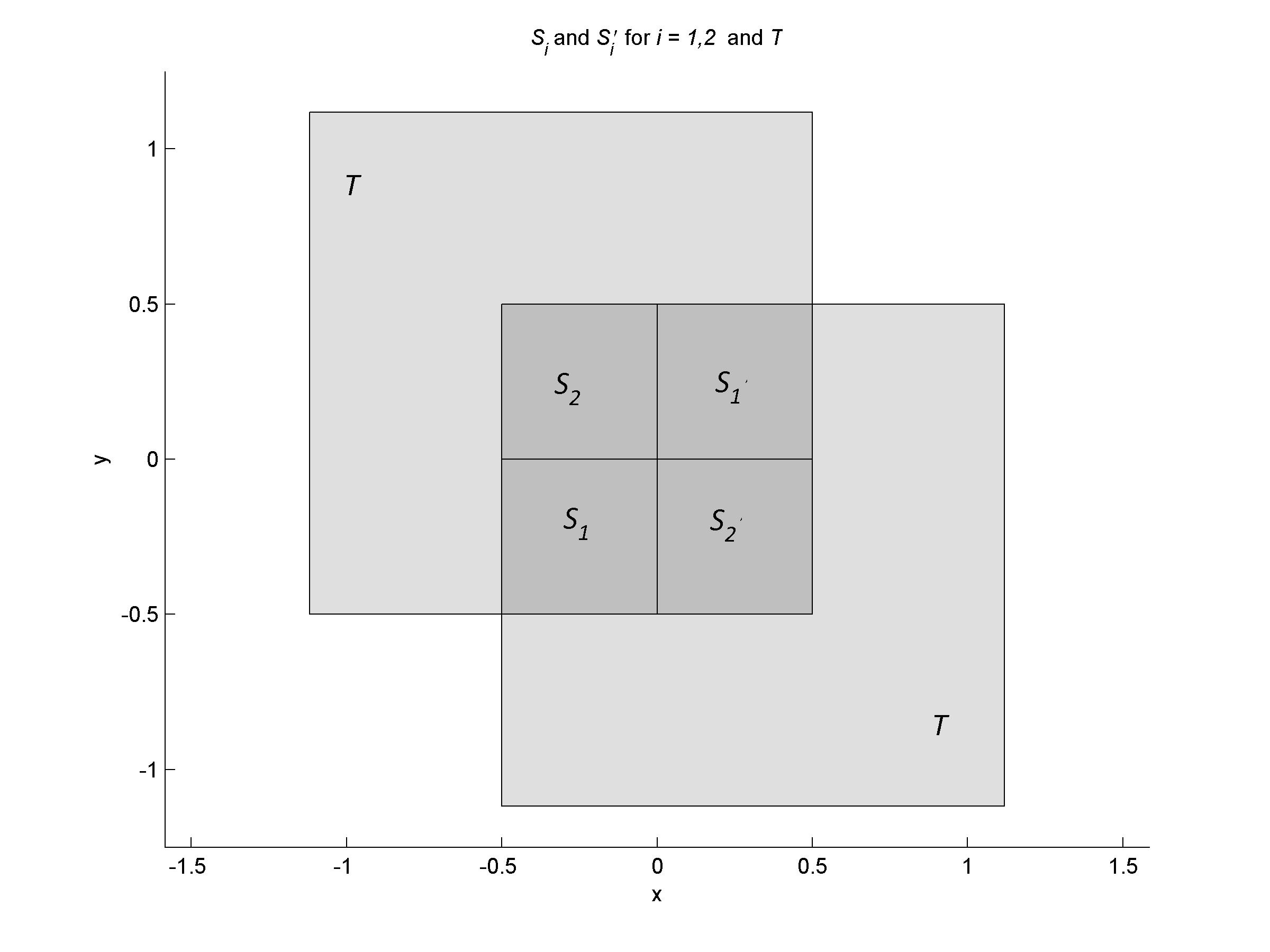}
\caption{$S_i$ and $S'_i$ for $i=1,2$}
\end{subfigure}
\caption{Partition of $R$}
\end{figure}

 We show now that the following mapping properties hold: \\
\[
  f(T_1\setminus \{p\}) \subset T_2' \cup T_3', \quad f(T_{2}) \subset T_2' \cup T_3', \quad   f(T_3 ) \subset S \cup T_1' ,\quad f(S) \subset S , \quad  f(\bar S) \subset \bar S, \quad f^2(\bar S ) \subset S\cup \{p_+, p_-\}
\]\\
from which it follows that  $f $ is forward invariant, i.e.  $f(R) \subset R$. The mapping properties result from the simple fact that  $-\frac{1}{4} \leq g(y) \leq \frac{1}{4}$ for the function $g(y) = -y^3 + \frac{3}{4}y$.\\

\begin{figure}[H]
\begin{subfigure}{0.28\textwidth}
\centering
\includegraphics[trim= 10cm 0mm 0mm 0mm,clip=true,keepaspectratio=true,scale=0.07]{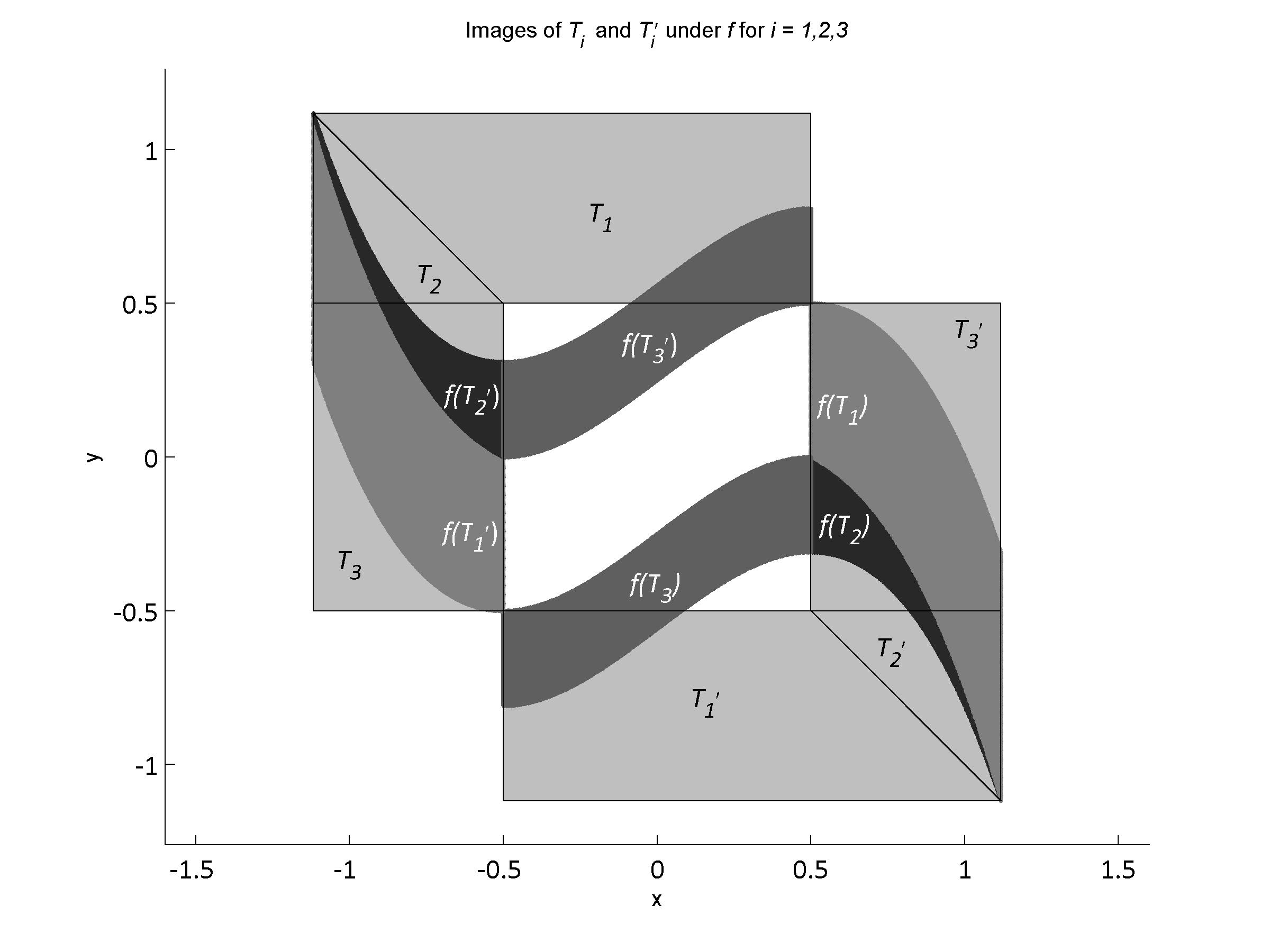}
\caption{$f(T_i)$ and $f(T'_i)$ for $i=1,2,3$}
\end{subfigure}\quad\quad
\begin{subfigure}{0.28\textwidth}
\centering
\includegraphics[trim= 7cm 0mm 0mm 0mm,clip=true,keepaspectratio=true,scale=0.07]{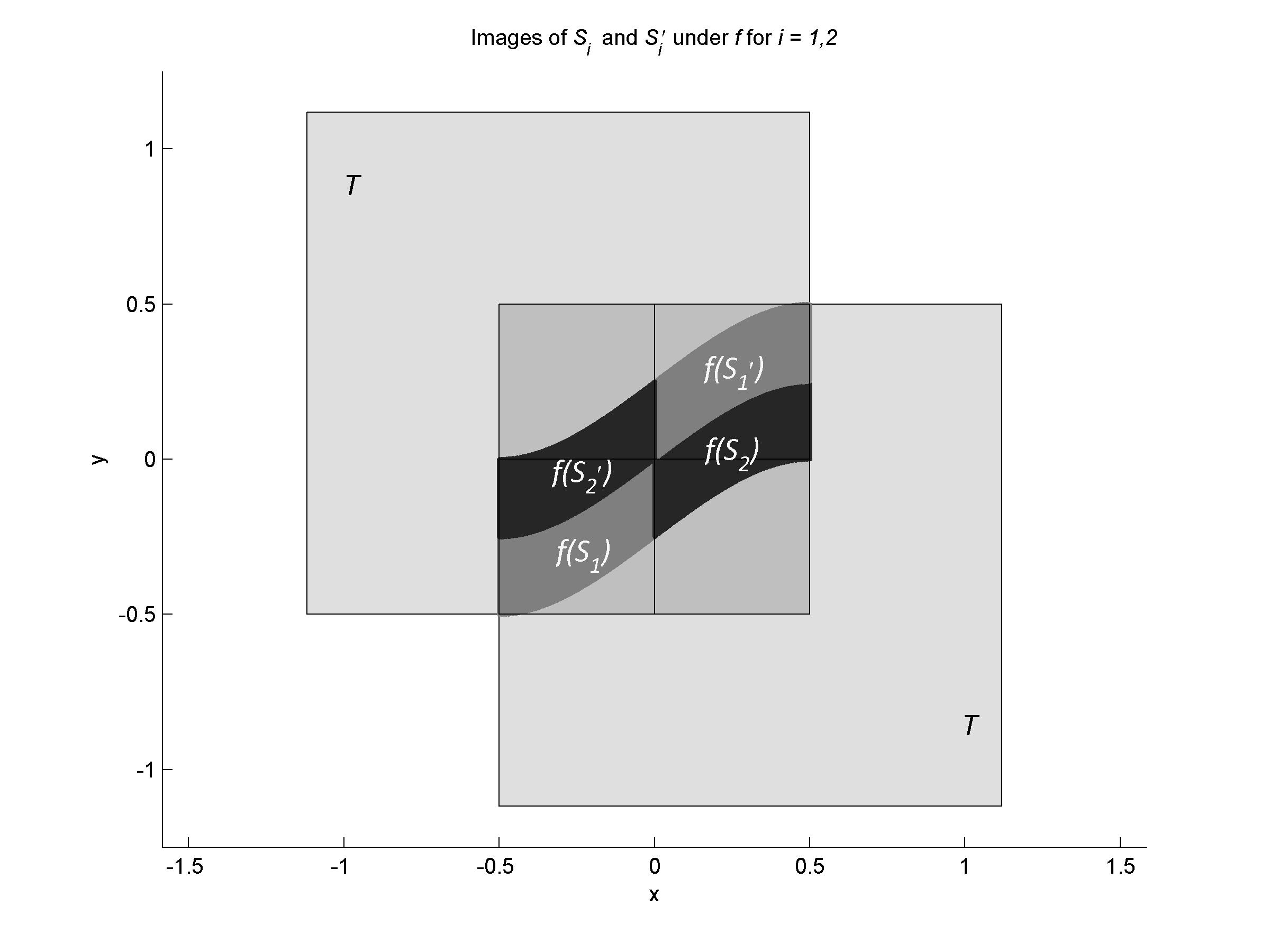}
\caption{$f(S_i)$ and $f(S'_i)$ for $i=1,2$}
\end{subfigure}
\caption{Forward Images of the Partition of $R$}
\end{figure}

 \textbf{Proposition 1}  $f(S)\subset S$ and $ f(\bar S) \subset \bar S$.

 \begin{proof} Since $S = (-\frac{1}{2},\,\frac{1}{2}) \times (-\frac{1}{2},\,\frac{1}{2})$ and $ f(x, y) = (x_1, y_1) = ( y, \frac{x}{2} + g(y))$, to show  $f(S)\subset S$ we need only note that $| y_1|<\frac{1}{2} $.  This implies $\overline{f(S)} \subset \bar S$. Futhermore, $S=f^{-1}(f(S) \subset f^{-1}(\overline{f(S)})$, which is a closed set since the pre-image of a closed set is closed under a continuous map.  Thus, $\bar S \subset f^{-1}(\overline{f(S)})$ which yields  $f(\bar S) \subset \overline{f(S)} \subset \bar S$.
\end{proof}

\textbf{Proposition 2} $f^2(\bar S) \subset S \cup \{p_+,\,p_-\}$

\begin{proof}  $f( S) \subset S$ implies that $f^2(S)  \subset S $.  Therefore,  we need only show that $f^2$ maps $\partial S$ into $ S \cup \{p_+,\,p_-\}$.  The boundary of $S$ is composed of four line segments: 
$$\ell_1 = \{ (x, y): |x| \leq \frac{1}{2}, y = \frac{1}{2}\}, \qquad \ell_2 = \{(x,y): |x| \leq \frac{1}{2},  y= - \frac{1}{2}\}$$
$$\ell_3 = \{(x,y):|y| \leq \frac{1}{2},  x= - \frac{1}{2}\},  \qquad \ell_4 = \{(x,y): |y| \leq \frac{1}{2}, x =  \frac{1}{2}\}$$
Denote $f(x, y) = (x_1, y_1) = (y, \frac{x}{2} + g(y))$. If $(x,y) \in \ell_1$ , then $x_1= \frac{1}{2}$ and $0\leq y_1\leq\frac{1}{2}$. Let $\ell'_4 = \{ (\frac{1}{2}, y): 0\leq y \leq\frac{1}{2}\}$. Thus, $f(\ell_1) \subset \ell'_4 \subset \ell_4$. When $(\frac{1}{2}, y) \in \ell'_4\setminus \{p_+\}$, then $0\leq y < \frac{1}{2}$ and  $\frac{1}{4}  \leq y_1 < \frac{1}{2}$, which means that $f(\ell'_4\setminus \{p_+\}) \subset S$ and therefore $f^2(\ell_1\setminus \{p_+\} )\subset S$. Since $\sigma (\ell_1) = \ell_2$, it follows that $ f^2(\ell_2\setminus \{p_-\}) \subset S$.\\
The images of $\ell_3$ and $\ell_4$ behave differently under $f$, namely except for the corners $(-\frac{1}{2}, \frac{1}{2})$ and $p_-$, they land in $S$ after one iteration. To see this, consider $\ell''_4 = \{ (\frac {1}{2}, y) : -\frac{1}{2} < y <0 \}$. If $(\frac{1}{2}, y) \in \ell'_4 $, then $g(y) <0$  and $ -\frac{1}{2} <  y_1  < \frac{1}{4} $, which implies $f(\ell''_4\setminus \{(\frac{1}{2}, -\frac{1}{2})\}) \subset S$.  Since $f^2(\frac {1}{2}, -\frac{1}{2}) \in S$, it follows that $f^2(\ell_4\setminus \{p_+\}) \subset S$. Because $\sigma (\ell_4) = \ell_3$, we have $ f^2(\ell_3 \setminus \{p_-\}) \subset S$ and Proposition 2 follows.

\end{proof}

\begin{center}
\begin{minipage}{0.4\textwidth}
	\centering
\begin{figure}[H]
	\centering
	\includegraphics[width=0.9\textwidth]{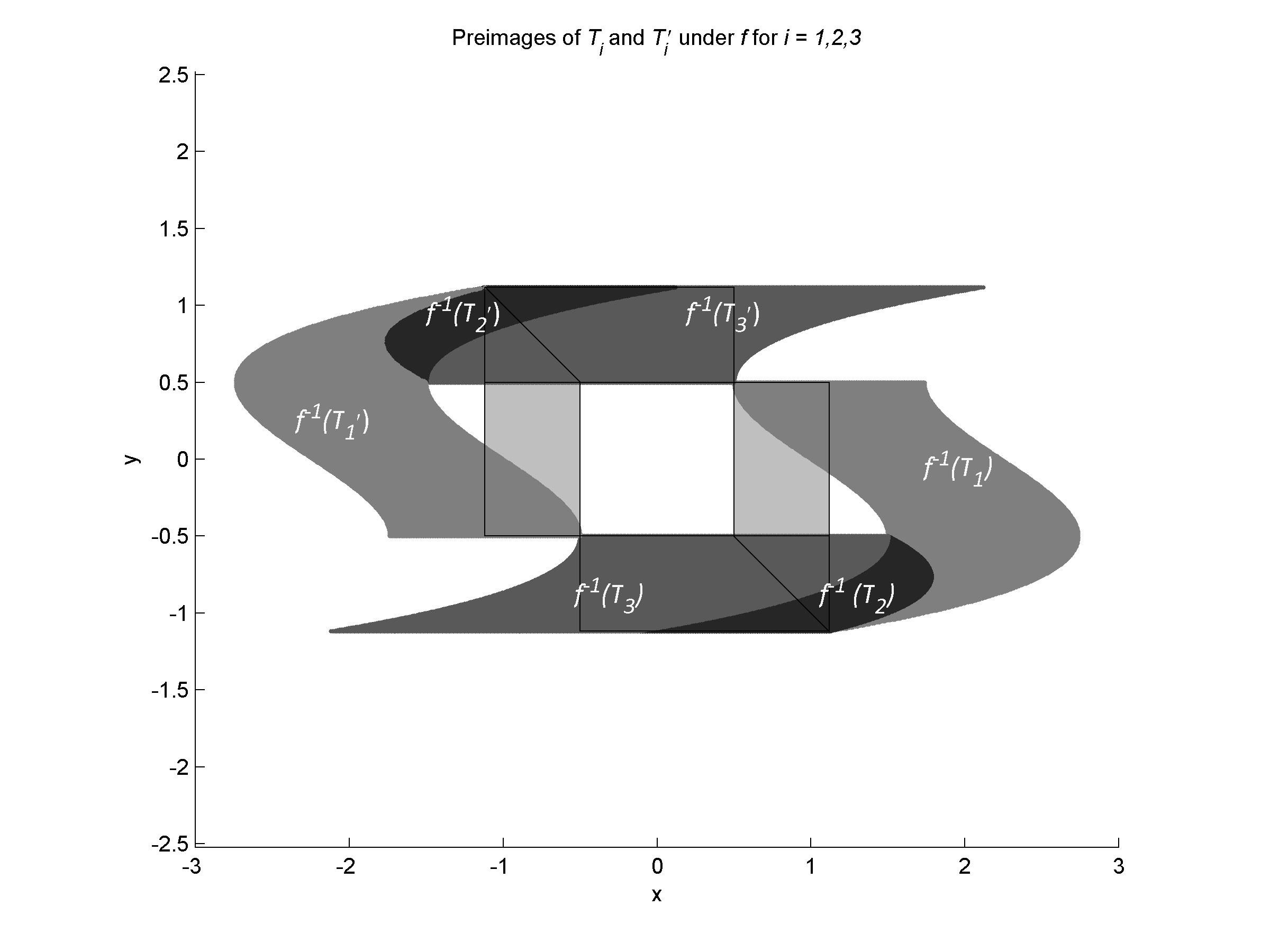}
	\caption{$f^{-1}(T_i)$ and $f^{-1}(T'_i)$ for $i=1,2,3$}
\end{figure}
\end{minipage}
\end{center}

\textbf{Proposition 3}  $f(T_1 \setminus  \{p\}) \subset T_2' \cup T_3 ' $,  \qquad $f(T_2) \subset T_2' \cup T_3 '$, \qquad $f(T_3) \subset S \cup T_1'$
\begin{proof}
Let $(x_1, y_1)= f(x,y)$. If $(x,y) \in T_1 \setminus \{p\}$,  obviously $x_1 =y  \in [\frac{1}{2}, \frac{\sqrt{5}}{2}]$.  It will  be enough to show that  $  -x_1   < y_1= \frac{x}{2} + g(y) \leq \frac{1}{2}$. Now $y_1\leq \frac{1}{2}$,  because $x \leq \frac{1}{2} $ and $g(y) \leq \frac{1}{4}$. For the lower bound, since  $-y \leq x$, we have $ y_1  \geq - \frac{y}{2}+ g(y) \geq -\frac{y}{2} + \frac{1}{4} > - y$ if and only if $y> \frac{1}{2}$. However, when $y = \frac{1}{2}, y_1 \geq -\frac{1}{4} + \frac{1}{4} = 0 > -y = -\frac{1}{2}$. $> -y$ when  $-y(y^2 - \frac{5}{4}) > 0$. That inequality is true for $y \in [\frac{1}{2},\frac{\sqrt{5}}{2}) $  but not for $ y = \frac {\sqrt{5}}{2}$. However, for $ y = \frac {\sqrt{5}}{2}$, $y_1 > -\frac{\sqrt{5}}{2}$ if and only if $x > -\frac{\sqrt{5}}{2}$, implying that $( x, y) $ cannot be $p$. Thus, $f(T_1 \setminus \{ p\}) \subset T_2' \cup T_3 '$. \\
Let $(x,y) \in T_2$.  Then, $ -\frac{\sqrt{5}}{2}\leq x \leq -y$ and $ \frac{1}{2} \leq y < -x$. Obviously, $\frac{1}{2} \leq  x_1 \leq \frac{\sqrt{5}}{2}$.   Next, we verify that $ -x_1 <  y_1 \leq \frac{1}{2}$.  Note that $ y_1 \geq -\frac{\sqrt{5}}{4} - y(y^2 - \frac{3}{4})\geq -\frac{\sqrt{5}}{4} -  \frac{y}{2} \geq -y $, because $|y^2 - \frac{3}{4}| \leq \frac{1}{2}$ and $y \geq \frac{\sqrt{5}}{2}$.  Therefore $(x_1, y_1) \in  T_2' \cup T_3 '$.\\
Let $(x,y)  \in T_3 $. Then $-\frac{\sqrt{5}}{2} \leq x \leq -\frac{1}{2}$ and $ |y| \leq \frac{1}{2}$. It suffices to show that   $-\frac{\sqrt{5}}{2} \leq y_1 \leq 0$. But $y_1 \leq -\frac{1}{4} + g(y) \leq 0$ and $y_1 \geq -\frac{\sqrt{5}}{4} + g(y) \geq  -\frac{\sqrt{5}}{4} - \frac{1}{4} \geq -\frac{\sqrt{5}}{2}.$

\end{proof}

The next Lemma treats all forward and all backward orbits of points in $T$.

\textbf{Lemma 2.3} The forward orbit of a point $q$ in $T$ which is not $p$ or $p'$ either eventually lands in $S$ or it stays in $T$ and converges to $p_+$ or $p_-$, i.e. $q \in \Omega_+ \cup \Omega_-$. The backward orbit of a point q in T which is not  $p_+$ or $p_-$ either eventually lands outside $R$ and escapes, i.e. $q\in W^u(\infty)$,  or it remains in T and converges to $\{p,\,p'\}$, i.e. $q\in W^u(p,\,p')$.

\begin{proof}
Note that the maximum norm $|(x, y)|$ is $|y| = y $  for $(x, y) \in T_1$ and $|(x, y)| = |x| = -x  $ for $(x, y) \in T_2 \cup T_3.$ Similar statements hold when $T_i$ is replaced by $T'_i$. Furthermore, when $(x, y) \in T_1$, then $|f(x, y)| = |(x, y)| = |y|$.  When $(x, y) \in T_2, $ then $|(x, y)| = |x| > |f(x, y)| = |y|$, because $|x| = -x , |y| = y$ and $ -x > y$ by definition of $T_2$. 

However, if $f(x, y) \in S$, the next Lemma will treat that  forward orbit. Consequently, after Proposition 3, we need only consider  $(x, y) \in T_3$ with $f(x, y) \in T'_1$. In that case, $|(x, y)| = |x| = -x$  and  $|f(x, y)| = -\frac{x}{2} + y^3 -\frac{3}{4}y $. Thus, $|(x, y)|  >  |f(x, y)| $ if and only if  $g(y) =   y(y^2 - \frac{3}{4}) < -\frac{x}{2} $. However, $ \frac{1}{4} \leq -\frac{x}{2} $ and $g(y) < \frac{1}{4}$ if $y > -\frac{1}{2} $.  When $ y = -\frac{1}{2}$,  then $|(x, y)| > |f(x, y)|$ if and only if $x < -\frac{1}{2}.$ Therefore, if $(x, y) \in T_3 \setminus \{p_-\} $, then $|(x, y)| > |f(x, y)|$. Similarly, if 
$(x, y) \in T'_3 \setminus \{ p_+\} $ with $ f(x, y) \in T_1$, then  $|(x, y)| > |f(x, y)|$ . To summarize, for all points $(x, y) \in T $ whose forward orbit remains in $T$, the sequence $(| f^n(x, y)| )_n$  is monotonically decreasing. Furthermore, the sequence of the norms of all even forward iterates $(|f^{2n}(x, y)|)_n$ is strictly monotonically decreasing for every point  $(x, y) \in T \setminus \{p_+, p_-\}$ It remains to show that $f^n(x,y) \to p_+ $ or  $f^n(x,y) \to p_-$.
Since $|f^n(x, y)| \geq \frac{1}{2}$ for every $n\in \mathbb{N}$,  this sequence converges. Let $r= \lim_{n \to \infty} |f^n(x, y)|$. Then $r\geq \frac{1}{2}$. 

We will see now that $r=\frac{1}{2}$. Because $T$ is compact, the forward iterates $f^n(x, y)$ have at least one accumulation point $a$. Then $|a|$ is an accumulation point of the convergent sequence $(|f^n(x, y)|)_n$ and $|a| = r$. Since $f(a)$ and $f^2(a)$ are also accumulation points of $f^n(x, y)$, we have $|a| = |f(a)| = |f^2(a)| = r$. If $r> \frac{1}{2}$, then every accumulation point $a$ would lie in $T$ and the contradiction $|f^2(a)| < |a| = r = \frac{1}{2} $ would follow.\\
Thus, every accumulation point $a$ must lie on the boundary of $S $. The only possible accumulation points however are  $p_+$ and $p_-$, since otherwise $f^2(a) \in S $ as noted above in Proposition 2. The backward invariance of the two basins $\Omega_+, \Omega_-$ implies then that $(x,y)$ must lie in one of them and consequently  $f^n(x,y) \to p_+ $ or  $f^n(x,y) \to p_-$.

Consider  a point q in T which is not  $p_+$ or $p_-$. We will investigate the behavior of the backward orbit  $O^-_f(q)=\{f^{-n}(q):\;n\in \mathbb{N}\}$  of $q$.  For $q\in T$, $O^-_f(q)\cap S=\emptyset$ because otherwise the forward invariance of S (Proposition 1) results in the contradiction $q \in S$.  Thus, either $O^-_f(q)\subset T$ or  there is an $n$ with $f^{-n}(q)$ not in $T$ and therefore not in $R$. If $f^{-n}(q)$  is not in $R$, it cannot be in $Q_3\cup Q'_3$, due to forward invariance, and therefore $f^{-n}(q) \to \infty$ by  Lemma 2.1.

Consider the case  $O^-_f(q)\subset T$.  The sequence $(|f^{-n}(q)|)_n$ will be shown to be monotonically increasing.  If $q= (x, y) \in T_2 \cup T_3$ and $f^{-1}(q) = ( x_{-1}, x)$ then obviously  $|f^{-1} (q)| \geq |x| = |q|$ by the definition of the maximum norm.  If  $q =(x,y)\in T_1$, then $ |q|=|y| = y$ and $f^{-1}(q)\subset T'_3 \cup (\mathbf{R^2}\setminus R)$. 
When $f^{-1}(q) \in T'_3$, then $|x| \leq \frac{1}{2} $ and $|f^{-1}(q)| = |x_{-1}| = x_{-1} = x( 2x^2 - \frac{3}{2}) + 2y. $ Consequently, $|f^{-1}(q)|  = x_{-1} \geq |q| = y $ if and only if $x(2x^2 - \frac{3}{2}) \geq -y $. But $-y \leq -\frac{1}{2}$,  and thus  $-\frac{1}{2} \leq x(2x^2 - \frac{3}{2})$, since $-x \geq -\frac{1}{2}$ and $ - (2x^2 - \frac{3}{2}) \geq 1$, proving the monotonicity. 

If $a$ is any accumulation point of $(f^{-n}(q))_n$, let $r= |a|$. Obviously, $r \leq \frac{\sqrt{5}}{2}$. Because $f^{-m}(a)$ is also an accumulation point, it follows that $ |f^{-m}(a)| = |a| = r$ for every $m \in \mathbb{N}$.  Hence, $a $ is on the boundary of $R$ 
and $r =  \frac{\sqrt{5}}{2}$, implying that $a$ is not $p_+$ or $p_-$. Due to 2.2, $a$ must be either $p$ or $p'$ proving the claim.\\
 \end{proof}

To treat the orbit behavior inside $S$, subdivide $S\setminus (0,0)$ into 4 open squares
 
 $$S_{1}=\{(x,y)\in  \mathbb{R}^2 : -\frac{1}{2} < x  < 0,\space - \frac{1}{2} < y < 0 \}, \qquad S_{2}=\{(x,y)\in \mathbb{R}^{2} : -\frac{1}{2} < x < 0, 0  <  y < \frac{1}{2} \},$$
$$S'_1 = \sigma (S_1),  \qquad  S'_2 = \sigma (S_2)$$.

The preimages of $S_i$ and $T_i$ under $f$ are depicted below:

\begin{center}
\begin{minipage}{0.4\textwidth}
	\centering
\begin{figure}[H]
	\centering
	\includegraphics[width=0.9\textwidth]{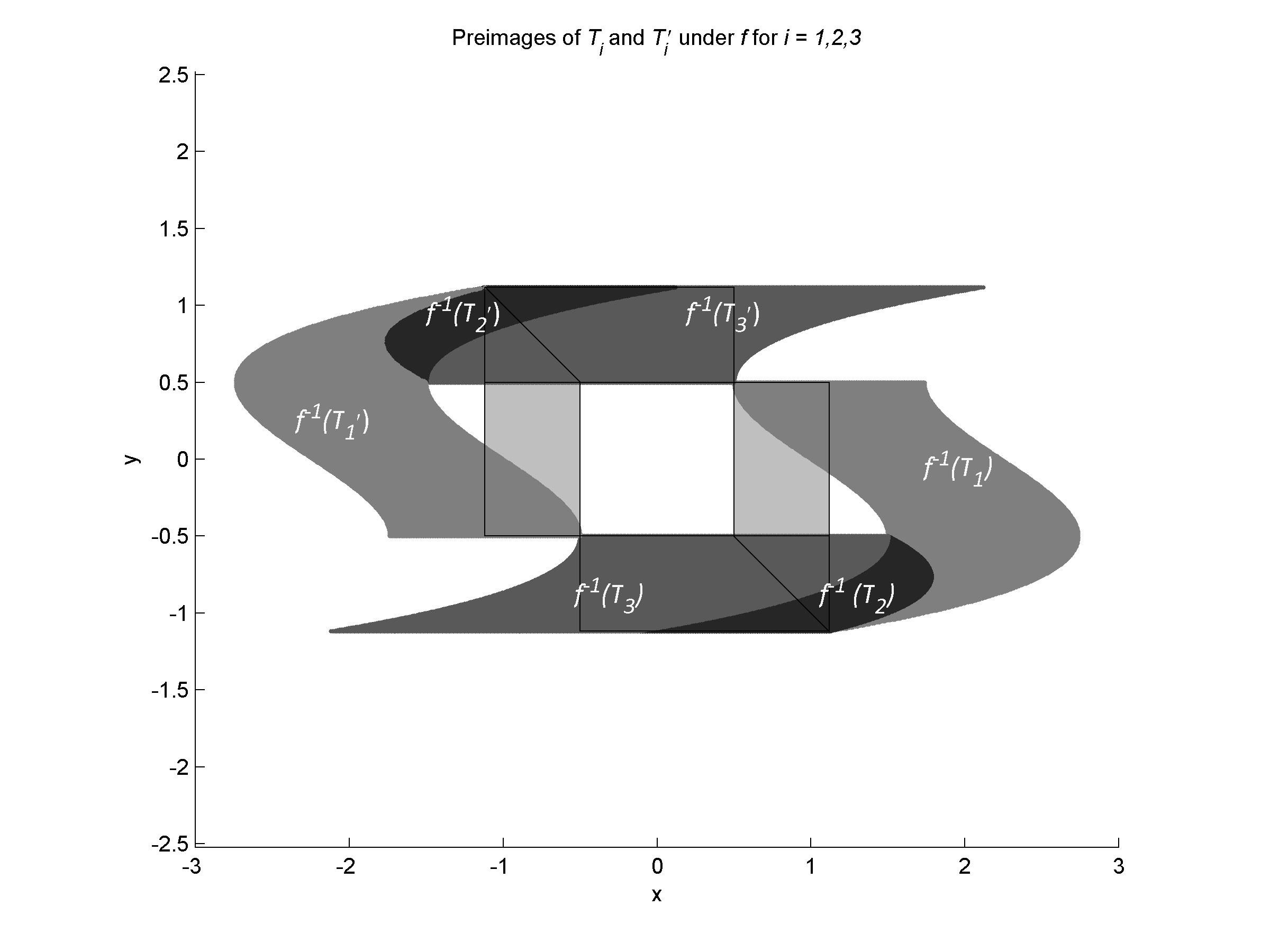}
	\caption{$f^{-1}(T_i)$ and $f^{-1}(T'_i)$ for $i=1,2,3$}
\end{figure}
\end{minipage}\qquad
 \begin{minipage}{0.4\textwidth}
	\centering
\begin{figure}[H]
	\includegraphics[width=0.9\textwidth]{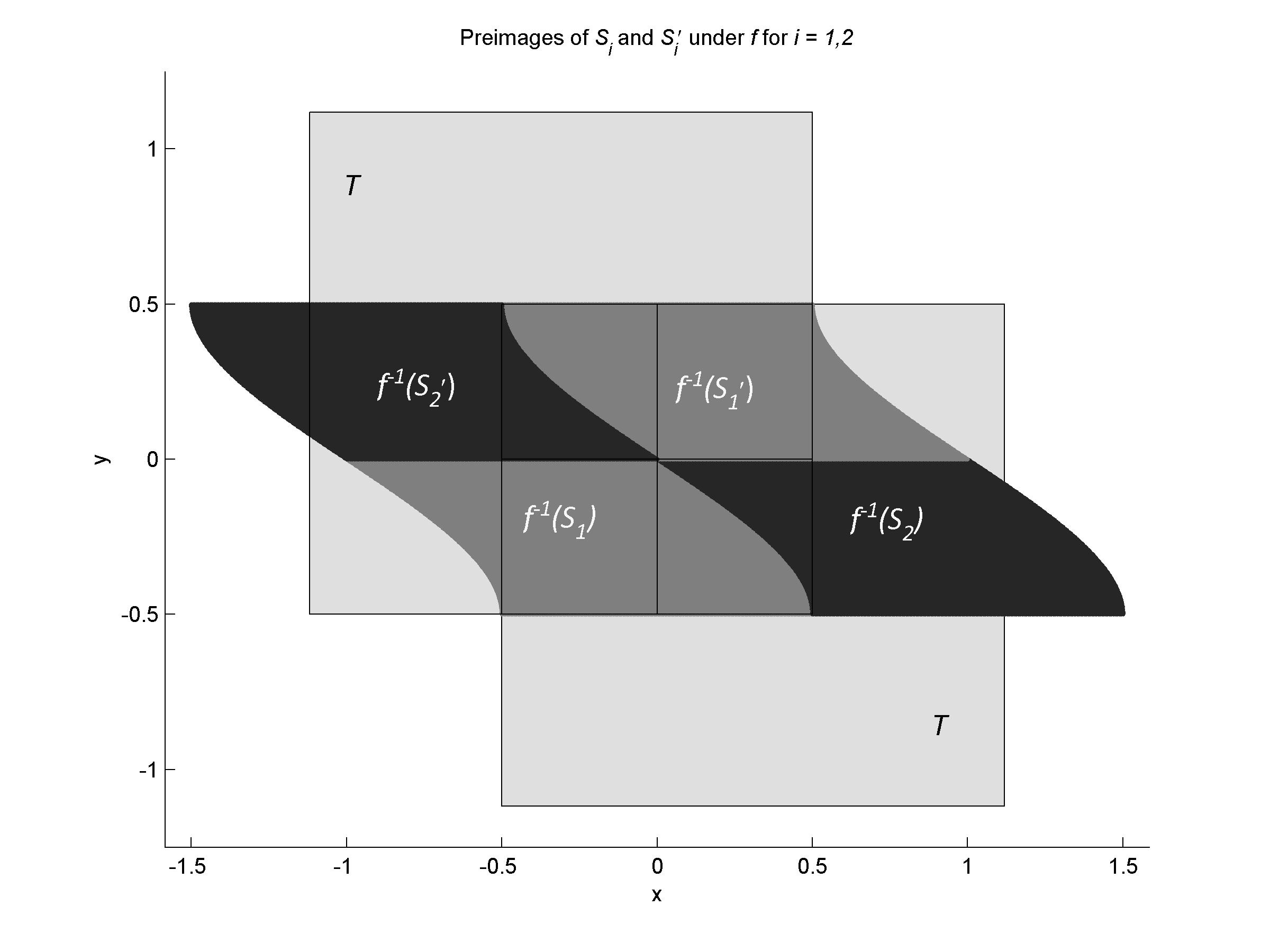}
	\caption{$f^{-1}(S_i)$ and $f^{-1}(S'_i)$ for $i=1,2$}
\end{figure}
\end{minipage}
\end{center}
        
  \textbf{Proposition 4} \qquad $f(S_1) \subset S_1$, \qquad $f(S_2) \subset S'_1 \cup S'_2$\\  
     \qquad $f^{-1}(S_1) \subset S_1 \cup S_2' \cup T_3$, \qquad  $f^{-1}(S_2) \subset S'_2 \cup T'_3 \cup (\mathbb{R}\setminus R)$
 
 \begin{proof}  
Let $(x, y) = S_1$. Clearly, $\frac{x}{2} +y(\frac{3}{4} - y^2) < 0$, since $\frac{3}{4} - y^2 > \frac{1}{2}$. From $-  \frac{1}{2} - \frac{x}{2}  < -\frac{1}{4}  < \frac{y}{2} < y (\frac{3}{4} - y^2) $, it follows that $\frac{x}{2} +y(\frac{3}{4} - y^2) > -\frac{1}{2}$. and $f(x, y) \in S_1$.
If $(x, y) \in S_2$, then $f(x, y) $ is in the right half plane. Due to Proposition 1, $f(S) \subset S$, implying $f(S_2) \subset S'_1 \cup S'_2$.

To see that  $f^{-1}(S_1) \subset S_1 \cup S_2' \cup T_3$,  let $(x,y)$ be a point in $S_1$ and let $f^{-1}(x,y)=(x_{-1},y_{-1})$.  It is clear that $-\frac{1}{2} < y_{-1} = x < 0$.  We  note that $-1  <  x_{-1} = 2y +2x^3 -\frac{3}{2}x < \frac{1}{2}$.  This is because   $-\frac{1}{2} < x < 0$ and $2x^2 - \frac{3}{2} <  -1$ imply $x_{-1} < \frac {1}{2}$ whereas $-1 < 2y < 0 < 2x^3 - \frac{3}{2} x $ implies $ x_{-1} > -1$. 
To prove  $f^{-1}(S_2) \subset S'_{2} \cup T_3 ' \cup (\R \setminus R)$, let $(x,y) \in S_2$.  It will be enough to show that $x_{-1} > 0$  which follows from $x_{-1} \geq x(2x^2 -\frac{3}{2})$ due to $x<0$ and $x^2 <\frac{1}{4}$.

 \end{proof}
 
Obviously,  the corresponding statements are true if $S'_i$ is interchanged with $S_i$ and $T'_3$ with $T_3$.

 The axes inside $S$ are mapped into $S_1 \cup S'_1$ after at most two forward iterations:\\
 
 \textbf{Proposition 5} $f(\{(0, y): 0<y <\frac{1}{2}\}) \subset S'_1, \qquad f(\{(0, y): -\frac{1}{2} < y < 0\}) \subset S_1$,\\
\qquad $ f^2(\{(x, 0): 0<x< \frac{1}{2}\}) \subset S'_1,\qquad  f^2(\{(x, 0): -\frac{1}{2} < x < 0\}) \subset S_1.$
 \begin{proof}
The positive $y-$axis in $S$ is mapped into $S'_1$,  because  $ 0<y <\frac{1}{2}$, implies $0 < y(\frac{3}{4} - y^2) < \frac{1}{2} (\frac{3}{4} - y^2) <  \frac{1}{2}$. The negative $y$-axis in $S$ is mapped into $S_1$ due to $\sigma \circ f = f \circ \sigma$.  The negative $x-$ axis in $S$ is mapped into the negative $y-$ axis in $S$ by $f$ and thus after another iteration it is mapped into $S_1$. Similarly, the positive $x-$ axis in $S$ is mapped into $S'_1$ after two iterations.

 \end{proof}

 All forward and all backward orbits of points in $S$ are treated next:
         
\textbf{Lemma 2.4 }The forward orbit of a point q in S stays in S and converges to either $(0, 0), p_+$ or $p_-$, i.e. $q \in W^{s}(0) \cup \Omega_-\cup \Omega_+$. The backward orbit of $q$ either eventually leaves $S$ for $T$ or it remains in $S$ and converges to the origin i.e. $q\in W^u(0)$. Furthermore, $\Omega_+$ contains $S'_1$ and $S_1$ is in $\Omega_-$.
\begin{proof}  First consider points $q$ in $S_1 \cup S'_1$. We will show that  $f^n(q)  \to p_+$ for $q \in S'_1$, from which $f^n(q) \to p_-$ for $q \in S_1$ follows,  due to $ \sigma \circ f = f \circ \sigma$. Let $q = (x, y) \in S'_1$. Using the pseudonorm $|(x, y)| = y + \frac{x}{2}$,
$$|f(x,y)| - |(x,y)| = \frac{x}{2} - y^3 + \frac{3}{4}y  + \frac{y}{2} - y - \frac{x}{2} = -y(y^2 - \frac{1}{4} ) > 0$$.
By induction,  the sequence $(|f^{n}(q)|)_n$ is strictly increasing. It is obviously  bounded, therefore it must converge.  Let $r $ denote the limit. Let $a$ denote an accumulation point of the forward orbit $(f^n(q))_n$. Then $|a| = r$. Since $f^m(a)$ is also an accumulation point for every $m$, it follows that $|a| = |f^m(a)| = r$ for every $m$, implying that $a$ must be on the boundary of $S'_1$, which means that $a$ is on the boundary of $T$ or on the positive $x-$ or $y-$ axis in $S$. The latter case cannot happen, because then  $a$ would be mapped into $S'_1$ after two iterations by Proposition 5. Hence, $a$ must be in $T$,  and Lemma 2.3 then shows that $f^n(q) \to p_{+}$.

Consider now points $q$ in $S_2 \cup S'_2$. Since forward orbits landing in $S_1 \cup S'_1$ have already been treated, because of Proposition 4 it suffices to look at points  $q \in S_2$ with $f^{2n}(q) \in S_2 $ for all $n\in \mathbb{N}$ and show that $f^{2n}(q) \to (0, 0)$.

Using the pseudonorm  $ |(x, y)| = |y - \frac{x}{2}|$, we have  $|f(q)| - |q|  = y(y^2 - \frac{5}{4}) $ for $q = (x, y)$ which is negative for $q = (x, y) \in S_2 $ with $f(q) \in S'_2$ and for $q\in S'_2 $ with $f(q) \in S_2$.  Therefore, $(|f^{n}(x,y)|)_n$ is strictly monotonically decreasing, as is $(|f^{2n}(q)|)_n$. Let $r = \lim_{n \to \infty}|f^{2n}(q)|$. Then $r \geq 0$. Since $|a| = |f^2(a)| = r $ for every accumulation point $a$ of the forward orbit $(f^{2n}(q))_n$,  $a$ cannot be in $S_2 $ and must be on $ \partial S_2 $. Due to Proposition 5, $a$  must be the origin or on $\partial S_2 \cap \partial S$. But if $a = (x, \frac{1}{2}), -\frac{1}{2} \leq x \leq 0$, then $|a| \geq \frac{1}{2},$ implying that $a$ cannot be an accumulation point. Similarly, if $a= (-\frac{1}{2}, y), 0 \leq y  \leq \frac{1}{2}$, then $|a| \geq \frac{1}{4}$ which means that such an $a$ also cannot be an accumulation point. Consequently, $a $ is the origin, $r = 0 = \lim _{n \to \infty}| f^{2n}(q)|$.  Then $f^{2n + 1}(q) \to (0, 0)$, since the origin is a fixed point for $f$, and hence $f^n(q) \to (0, 0)$ follows.

We turn to the backward orbits of points $q$ in $S$. It suffices to consider the two cases: $O^-_f(q)\subset S_1$ and $O^-_f(q)\subset S_2 \cup S'_2$ .  

Suppose now that $ q = (x,y) \in S_1$ and $f^{-n}(q) \in S_1$ for all $n\in \mathbb{N}$.  Letting $|(x,y)|=|y + \frac{x}{2}|$, we show that the sequence $(|f^{-n}(q)|)_{n \in \mathbb{N}}$ is strictly monotonically decreasing and converges to $0$.  To see this, we note that for all such $(x,y) \in S_1$, the following inequality holds:

\begin{align*}
 |f^{-1}(x,y)| - |(x,y)| & =  |\frac{1}{4}x + x^3 + y| -|y+\frac{x}{2}|\\ 
&= (-\frac{1}{4}x - x^3 -y) - (-y -\frac{x}{2}) \\
&=\frac{1}{4}x - x^3<0 \\
\end{align*}

which is true for $-\frac{1}{2}< x < 0$.  As the sequence $(|f^{-n}(q|)_{n \in \mathbb{N}}$ is strictly monotonically decreasing and bounded from below by zero, we conclude that $f^{-n}(q) \to (0,0)$.  Otherwise $(|f^{-n}(q)|)_{n \in \mathbb{N}}$ would converge to some constant $r>0$.  If the point $a$ is an arbitrary accumulation point for the backward orbit $(f^{-n}(q))_{n \in \mathbb{N}}$, then  $|a|=r$.  Since $f^{-1}(a)$ is also an accumulation point, it follows that $|f^{-1}(a)|=|a|=r$ and $a$ cannot be in $S_1$, implying that $a$ is on $\partial S_1$. By Proposition 5, $a$ cannot be on the negative $x-$ or $y-$ axes, and thus $a$ is either the origin or on $\partial S$. The latter situation would mean  $f^{-n}(a) \in \partial S$ for all $n$, contradicting Lemma 2.3 which states that  $a \in W^u(p, p')$. Therefore,  $f^{-n}(x,y) \to (0,0)$.

Using the pseudonorm defined by  $|y - \frac{x}{2}|$,  we now prove that $|f^{-1}(x,y)| - |(x,y)|> 0$  for every point $(x,y)\in S_2 \cup S_2'$.   Without loss of generality, suppose $(x,y)$ is in $S_2$.   By definition $|f^{-1}(x,y)|=  |\frac{7}{4}x - x^3 -y|$.  Since $\frac{7}{4}x - x^3 -y$ is negative and $y-\frac{x}{2}$ is positive for all relevant values of $x$ and $y$,  the result will follow if $- (\frac{7}{4}x - x^3 -y) > y-\frac{x}{2}$, i.e. $-\frac{7}{4}x +x^3 > -\frac{x}{2}$ which is true for all $x$ in the interval $(-\frac{1}{2}, 0)$.  Thus, $(|f^{-n}(x,y)|)_{n \in \mathbb{N}}$ is strictly monotonically increasing for all $(x,y) \in S_2 \cup S_2'$.  From this we may deduce that $(S_2 \cup S_2') \cap W^{u}(0) = \emptyset$. \\
 
\end{proof}

\textbf{Remark} \qquad $S \subset W^s(0)\cup\Omega_+\cup\Omega_- $, \qquad $S \subset W^u(\infty) \cup W^u(p, p') \cup W^u(0)$,\\
 \qquad $W^s(0) \cap S \subset S_2 \cup S'_2 $, \qquad $  W^u(0) \cap( S_2 \cup S'_2) = \emptyset$ \\

Combining Lemmas 2.3 and 2.4, we know the fate of every forward and every backward orbit of points in $R$:\\

\textbf{Lemma 2.5} The forward orbit of a point $q$ in $R$, which is not $p$ or $p'$, converges to $0$, $p_+$, or $p_-$, i.e. $q\in W^s(0)\cup\Omega_+\cup\Omega_-$. The backward orbit of a point $q$ in $R$  different from $p_+$ or $p_-$  converges to $0$, $\{p,p'\}$, or escapes, i.e. $q\in W^u(0)\cup W^u(p,p') \cup W^u(\infty)$.\\

In particular, we now know that $R$ is contained in the set $K_\R^+$ of real points with bounded forward orbits. Moreover, Lemma 2.1 implies the following result contained in [K, Proposition 7.10]:\\

\textbf{Remark} There are only 5 real periodic points for $f$, namely the three fixed points  $(0,0), p_+, p_-,$ and the period 2 cycle $ p,  p'$.\\

\maketitle
\section{STABLE AND UNSTABLE MANIFOLDS}

There is an unstable eigenvector $v$ of $D\,f^2(p)$ pointing into the interior of $Q_3$ and a parameterization $\Gamma:\;\R\to\R^2$ of the unstable manifold $W^u_{f^2}(p)$ of $p$ with respect to $f^2$ having $v$ as its tangent at $p$. A similar statement holds for a parameterization $\Gamma'$ of $W^u_{f^2}(p')$ replacing $p$ by $p'$ and $Q_3$ by $Q'_3$. Denote $\Gamma(\R_-)\cup\Gamma'(\R_-)$ by $W^u_-(p,p')$ for the negative real numbers $\R_-$and $W^u_+(p,p'):=\Gamma(\R_+)\cup\Gamma'(\R_+)$. Note that $W^u(p,p')=W^u_-(p,p')\cup W^u_+(p,p')\cup\{p,p'\}$.

\textbf{Theorem 3.1}

(i) $ W^u(0)\subset S_1\cup S'_1\cup \{0\} $\\
(ii) $\overline{W^u(0)}=W^u(0)\cup \{p_+,p_-\} $\\
(iii) $W_-^u(p,p')\subset int R\;and\;W_+^u(p,p')\subset Q_3\cup Q'_3$\\
(iv) $ W^u_-(p,p')\not\subset W^s(0)$\\
(v) $ W^s(p,p')\cap R=\{p,p'\}$

\begin{proof} $(i)$ If $q \not \in R$ and $q \not \in Q_3 \cup Q'_3$, then  $f^{-n}(q) \to \infty$ follows from  Lemma 2.1 and therefore $ q \not \in W^u(0)$ showing that $W^u(0) \subset R \cup (Q_3 \cup Q'_3)$.  Because $W^u(0)$ is connected, if there was a point in $W^u(0) \cap( Q_3 \cup Q'_3)$, then $p$ or $p'$ would have to be on $W^u(0)$, a contradiction. Thus, $W^u(0) \subset R$.  However, if $q\in W^u(0) \cap T$, Lemma 2.3 implies that $q \in W^u(p, p')$ would  result, a contradiction, and consequently, $W^u(0) \subset S$. By the above remark, $ W^u(0) \cap (S_2 \cup S'_2) = \emptyset$  and (i) follows.

$(ii)$ If $p_-$ is a limit point of the unstable manifold $W^u(0)$, then by symmetry ( using the reflection map ),  $p_+$ is also a limit point. To show that $p_-$ is a limit point, take any point $q \in W^u(0)$ which is not the origin. By $ (i)$ it is no restriction to assume that $q \in S_1$. After Lemma  2.4, $f^n(q) \to p_-$. The invariance of the unstable manifold implies that $f^n(q) \in W^u(0)$ and hence $p_- \in \overline{W^u(0}$. To see that $p_-$ and $p_+$ are the only limit points, let $q \in L = \overline{W^u(0)} \setminus W^u(0)$. By (i),  $q \in \overline{S_1} \cup \overline {S'_1}$. Without restriction, let $q \in \overline {S_1}$. Then 2.5 implies that $q \in \partial S_1$, because otherwise  $q \in  W^u(\infty) \cup W^u(p, p')$. Since the set $L$ of limit points is closed and invariant,  $f^{-n}(q) \in L$ for all $n$, implying that $q \notin W^u(\infty)$ and thus $q \in W^u(p, p')$ must follow. This in turn would mean that $\{p, p'\} \subset L \subset \overline{S}$ which is a contradiction. If $q \in \partial S_1$ but $q$ is not $p_-$, then by 2.3 it would follow that $q \in W^u(p, p')$  if $q$ was not on an axis in $S$ and once again the same contradiction. If $q$ was on an axis in $S$, then $f^2(q) \in S_1$ by Proposition 5 and the contradiction $q \in W^u(p, p')$ again follows from 2.5.

$(iii)$  Let $J_{f^2}(p)$ and $J_{f^2}(p')$ represent the Jacobian matrices of $f^2$ at $p$ and $p'$ respectively.  Then, by the chain rule, $J_{f^2}(p)=J_f(p)\cdot J_f(p')=J_{f^2}(p') = $

\[\begin{bmatrix}  \frac{1}{2} & -3 \\  -\frac{3}{2}  & \frac{19}{2}
\end{bmatrix}
\]
which yields the unstable eigenvector $e_u=\langle -.32, 1 \rangle$.  Hence, by the unstable manifold theorem, $W^u_{f^2}(p)$ (resp. $W^u_{f^2}(p')$) is tangent to the parameterization $\Gamma:\mathbb{R}\to\mathbb{R}^2$ defined by $t \mapsto p+t\cdot e_u$ (resp. $\Gamma':\mathbb{R}\to\mathbb{R}^2$ defined by $t \mapsto p'-t\cdot e_u$).  Thus, it is clear that  $W_-^u(p,p')=\Gamma(\mathbb{R}_-)\cup\Gamma'(\mathbb{R}_-)\subset R$ by the forward invariance of $R$.  Furthermore, since for $q\in Q_3\cup Q'_3, \;f^n(q)\to\infty$ as $n\to\infty$, it follows that  $W_+^u(p,p')=\Gamma(\mathbb{R}_+)\cup\Gamma'(\mathbb{R}_+)\subset Q_3\cup Q'_3$.

$(iv)$  If  $ W^u_-(p,p')\subset W^s(0)$ were to hold, assume that $W^u_-(p, p') = W^s(0) \setminus \{0\} $ would follow (a fact that will be proved later). Now it will be shown that $W^u_-(p, p') = W^s(0) \setminus \{0\} $  implies a contradiction. 
This assumption means that   $W^s(0) \subset int R$ by $(iii)$.  However, a point $q \in \partial R \cap W^s(0) $ can be constructed giving the contradiction. Let $a = (\frac{\sqrt{5}}{2}, -\frac{1}{3})$. Obviously, $a \in \partial R$. Furthermore,  
$f^2(a) \in S_1'$ which means by Lemma 2.4  that  $f^2(a) \in  \Omega_+$. Consequently, $a\in \Omega_+$ due to the invariance of basins of attraction. Let $b = (\frac{\sqrt{5}}{2}, -\frac {33}{32})$. Obviously, $b\in \partial R$. A calculation shows that $f^5(b) \in S_1$  and therefore  $b \in \Omega_-$, again because of 2.4 and invariance. \\
The line segment connecting $a$ and $b$ on $\partial R$  intersects  $\Omega_+$ as well as its complement and therefore must contain a point $q \in \partial
\Omega_+$. Since $q $ is in $R$ but neither in $\Omega _+$ nor in $\Omega_-$, 
by 2.5 $q\in W^s(0)$ must follow.\\
The final step is to prove that  $W^{u}_{-}(p,p') \subset W^s(0)$ implies $W^{u}_{-}(p,p')=  W^{s}(0) \setminus \{0\}$.  We denote the two path components of $W^s(0) \setminus \{0\}$ by $C_1$ and $C_2$.
Let $\Gamma: \R \to \R^2$ and $\Gamma ': \R \to \R^2$ denote the parametrizations of $W^u_{f^2}(p)$ resp. $W^u_{f^2}(p')$.  By definition, $W^{u}_{-}(p,p') = \Gamma(\R_-) \cup \Gamma ' (\R_-)$.  By our assumption, $\Gamma(\R_-),$ and $\Gamma ' (\R_-)$ are path connected subsets of $W^s(0) \setminus \{0\}$. Since  $\Gamma(\R_-)$ and    $\Gamma ' (\R_-)$ are both forward and backward invariant under $f^2$, we may conclude without restriction that $\Gamma(\R_-)= C_1$ and $\Gamma'(\R_-)= C_2$.

$(v)$ By Lemma 2.5, the forward orbit of a point $q$ in $R$ which is not $p$ or $p'$ converges either to $0$, $p_+$ or $p_-$.  In other words, such a point $q$ does not belong to $W^s(p, p')$.  Hence, $W^s(p, p') \cap R = \{p, p'\}$.    

\end{proof}

Note that Lemma 2.5 and Theorem 3.1(i) imply the\\

\textbf{Remark} $W^u(0)\setminus\{0\} \subset \Omega_+ \cup \Omega_-$\\


\maketitle
\section{JULIA SETS}

The real filled Julia sets $K_\R:= K\cap \R^2,\;K^+_\R:=K^+\cap \R^2,\;K^-_\R:=K^-\cap\R^2$ and the real Julia sets $J^+_\R:=J^+\cap\R^2,\;J^-_\R:=J^-\cap\R^2$ and $J_\R:=J\cap\R^2$ for $f$ can now be  calculated in terms of the stable and unstable manifolds of the 5 periodic points $\{0,p_+,p_-,p,p'\}$.
Note that $K^+ $ and $K^-$ are closed and $K$ is compact in $\C^2$ by [FM]. By definition, $J^+= \partial K^+, J^- = \partial K^-$ and $J= J^+\cap J^-$. Note also that 
$K^+_{\R} = W^s(K_{\R}) = \{ q \in \R^2 : d(f^n(q), K_{\R}) \to 0\}$ and $K^-_{\R} = W^u(K_{\R}) = \{ q \in \R^2 : d(f^{-n}(q), K_{\R}) \to 0\}$ where $d$ denotes the Euclidean metric in $\R^2$, see [BS]. In [BS] it is  shown that $J^+ $ is the closure in $\C^2$ of the stable manifold of any saddle point in $\C^2$ and  $J^+ $ is also the boundary in $\C^2$ of every complex basin of attraction.


\textbf{Theorem 4.1}
\begin{align*}
(i &) \; K_\R = \overline{W^u(0)} \cup W^u_-(p,p') \cup \{p,p'\}\\ 
(ii &) \; K_\R^+=\Omega_+\cup\Omega_-\cup W^s(0)\cup W^s(p,p') \\
(iii &) \; K^-_\R=\overline{W^u(0)}\cup W^u(p,p')\\
(iv &) \;J^+_\R=W^s(0) \cup W^s( p,p') \\
(v &) \;J^-_\R=K^-_\R \\
(vi &) \;J_\R=\{0,p,p'\}\cup(W^u_-(p,p')\cap W^s(0))
\end{align*}

\begin{proof}  (i)   Clearly,  $K_{\R} \supset \overline{W^u(0)}  \cup W^u_-(p,p') \cup \{p, p' \}$,  since $W^{u}(0) \subset R$, implying  $\overline{W^u(0)} \subset R,$ and $W^u_-(p,p') \subset R$ by Theorem 3.1 and because $R \subset K_{\R}^+$  by Lemma 2.5. \\
For the reverse inclusion,  let $q$ be a non-periodic point in $K_{\R}$.  As  $K_{\R}$ is invariant, the backward orbit of $q$ stays in $K_\R$ hence in $R$ due to Lemma 2.1.  Then, $ q \in W^u(0) \cup W^u(p, p')$ by Lemma 2.5.  \\
(ii) The inclusion of the right hand side of the equation in the left hand side is immediate.
Let $q \in K_\R^+$. If the forward orbit $O^+_f(q)$ eventually lands in $R$, then $q\in \Omega_+\cup \Omega_- \cup W^s(0)$ by Lemma 2.5. Now let  $O^+_f(q) \cap R = \emptyset$. Then $ d(f^n(q), K_{\R}) \to 0$ and $q \notin \{p, p', p_+, p_-\}$. But $K_{\R} \setminus \{p, p', p_+, p_-\}$ is in the interior of $R$ by Lemma 2.1, implying that $d(f^n(q), \{p, p', p_+, p_-\}) \to 0$. If $p_+ $ and $p_-$ are limit points of $O^+_{f}(q)$, then $q \in \Omega_+\cup \Omega_- $. If $q \notin \Omega_+\cup \Omega_- $, then $d(f^n(q), \{p, p'\}) \to 0$ and $q \in W^s(p, p')$.\\
(iii) The inclusion  $\supset$ is obvious. To show the opposite inclusion, let $q $ be a point in $ K_\R^-$  which is different from $p, p', p_+$  and $ p_-$. If $q \in R$, then $q \in K_\R$, since $R$ is forward invariant, and the claim follows from (i).  If $q \notin R$, then the entire backward orbit $O^-_f(q) $ is not in $R$. However, $K_\R^- = W^u(K_\R)$ implies  $d(f^{-n}(q), K_\R) \to 0$, and again from Lemma 2.1 it follows that $d(f^{-n}(q), \{p, p'\})\to 0$, i.e. $q\in W^u(p, p')$.\\
(iv) To prove the inclusion $\subset$, let $q\in J^+_\R $. Then $q\in K^+_\R$, since $ J^+_\R = \partial K^+ \cap \R^2 \subset K^+ _\R$ because $K^+$ is closed. On the other hand,  $q\in J^+ = \partial \Omega^C_+ = \partial\Omega^C_-$ by Theorem 2 of [BS],  where $   \Omega^C_+$ is the basin of attraction  of $p_+$ in $C^2 $ and $   \Omega^C_-$  that of $p_-$ . It follows that $q\notin \Omega_+$ and that $q\notin \Omega_-$. Part (ii) shows that $q \in W^s(0)\cup W^s(p,p') $. The opposite inclusion follows from Theorem 1 in [BS] which proves that the closure of the stable manifold in $C^2$ of any saddle point is $J^+$.\\
(v) The Jacobian determinant $det \,Df$ of $f$ is $-\frac{1}{2}$ which implies by [FM, Lemma 3.7] that the 4-dimensional Lebesque measure of $K^-$ is zero. Consequently, $K^-$  has no interior points and $J^- = \partial K^- = K^-$ gives the claim $J^-_\R = K^-_\R$.\\
(vi) Using (iv), (v), and (iii), 

\item $$J_\R =  (W^s(0) \cup W^s(p, p')) \cap (\overline{W^u(0)} \cup W^u(p, p'))$$
$$=(W^s(0) \cap \overline{W^u(0)}) \cup (W^s(p, p')\cap \overline{W^u(0)}) \cup (W^s(0)\cap W^u(p, p')) \cup (W^s(p, p')\cap W^u(p, p')).$$\\

We will treat the last intersection first and show that $W^s(p, p')\cap W^u(p, p') = \{p, p'\}$. Now $W^s(p, p') \cap (Q_3\cup Q'_3)
 = \{p,p'\}$,  because by Lemma 2.1 points in $Q_3 \cup Q'_3$ different from $p$ or $p'$ have orbits which escape to infinity  under forward iteration. Then $W^u_+(p, p') \subset Q_3 \cup Q'_3$ implies $W^s(p, p')\cap W^u_+(p, p') = \emptyset $. Because  $W^u_-(p, p')$ is in the interior of $R$, Theorem 3.1(v) shows that $W^s(p, p')\cap W^u_-(p, p') = \emptyset $.
 The second intersection is also the empty set by Theorem 3.1(v), since by 3.1(i) and (ii), $\overline{W^u(0)}$ is in $R$.
 The first intersection is equal to $\{0\}$, because $W^u(0)\subset S_1 \cup S'_1 \cup \{0\}$ by 3.1(i) and $S_1 \cup S'_1 \subset \Omega_+ \cup \Omega _-$ after 2.4.
 Finally, the third intersection is equal to $W^s(0) \cap W^u_-(p, p')$, due to $W^s(0) \cap \{p, p'\} = \emptyset$ and $W^s(0) \cap W^u_+(p, p') = \emptyset$, using the fact that  $W^u_+(p, p') \subset Q_3 \cup Q'_3$.
\end{proof}

\section{BASIN BOUNDARIES}

\begin{thm} $\partial\Omega_+=\partial\Omega_-=W^s(0)\cup W^s(p,p') =J^+_\R$
\end{thm}

\begin{proof}
It is clear that $\Omega_+$ and $\Omega_-$ both lie in $K^+_\R$. Then $\partial \Omega_+ \cup \partial \Omega_- \subset K^+_\R$ follows, since $K^+_\R$ is closed, implying by Theorem 4.1(ii)  that  $\partial \Omega_+ \cup \partial \Omega_- \subset W^s(0)\cup W^s(p,p')$. \\
To prove the opposite inclusion, it will first be shown that $W^s(0) \subset \partial \Omega_+$. Due to the invariance of the basin boundary,  it suffices to show that the local stable manifold of the origin $W^s_{\epsilon}(0)$  of size $\epsilon$ is contained in $\partial \Omega_+$. Without restriction, $W^s_{\epsilon}(0) \subset S_2 \cup S'_2$.  Let $q \in W^s_{\epsilon}(0) \cap S_2$, and let $U$ be a polydisk around $q$. Every point $q'$ in $U \setminus W^s(0) $ is in $ \Omega _+ \cup \Omega_-$ by 2.4. If $q'\in \Omega_+$, then  $q \in \partial \Omega_+$  and  $\sigma(q) \in \partial\Omega_-$. \\
It remains to show  $W^s(p, p') \subset \partial \Omega_+ $, since $W^s(p, p') \subset \partial \Omega_-$ follows due to $\sigma(W^s(p, p')) = W^s(p, p'))$ and $\sigma(\partial\Omega_+) = \partial \Omega_-$. By 3.1(iv), there is  a point $ q \in W^{u}_{-}(p,p^{'})$ such that $q \notin W^{s}(0)$. Then $q $ is in the interior of $R$ and by 2.5, $q \in \Omega_{+} \cup \Omega_{-}$. Without restriction, let $q \in \Omega_+$. An application of the Lambda Lemma  (see[R]) will be used. Take a curve $C$ through $q$ transversal to $W^u_-(p, p')$ which is contained in $\Omega_+$. Then parts of the backward iterates $f^{-n}(C)$  converge to the local stable manifold $W^s_{\epsilon}(p, p')$ of $\{p, p'\}$   in the $C^1$ topology proving $W^s(p, p') \subset \partial \Omega_+$.

\end{proof}

\begin{figure}[H]
\includegraphics[keepaspectratio=true,scale=0.2]{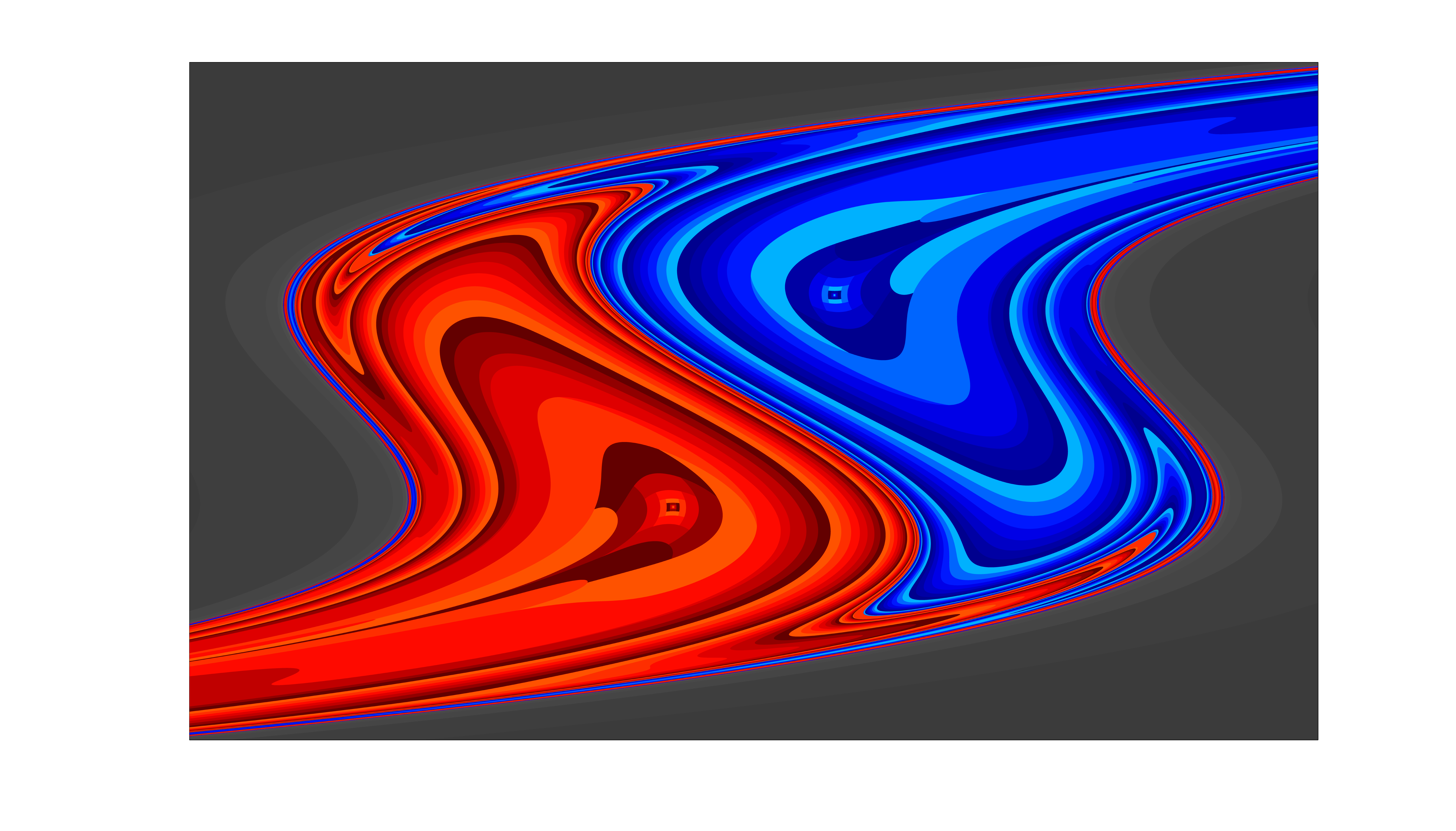}
\caption{The Real Section under Bieberbach's Map}
\end{figure}

\textbf{REFERENCES}\\

[BS] E. Bedford and J. Smillie, \textit {Fatou-Bieberbach domains arising from polynomial automorphisms}, Indiana Univ. Math. J. \textbf{40} (1991) 789-792.

[BS2] E. Bedford and J. Smillie, \textit{ Polynomial Automorphisms of $\C^2$. II: Stable Manifolds and Recurrence}, J. Amer. Math.Soc. \textbf{4} (1991)  657-679.

[FM] S. Friedland and J. Milnor, \textit {Dynamical properties of plane polynomial automorphisms}, Ergod. Th. and Dynam. Sys. \textbf{9} (1989), no.1, 67-99.

[FS] J.E. Fornaess and N. Sibony, \textit {Complex H\'enon mappings in $\C^2 $ and Fatou-Bieberbach domains}, Duke Math. J. \textbf{65} (1992), no.2, 345-380.

[H] S. Hayes, Fatou-Bieberbach Gebiete in  $\C^2$, Deutsche Mathematiker Vereinigung Mitteilungen, (1995), 14-18.

[K] T. Kimura, \textit{ On Fatou-Bieberbach domains in $\C^2$}, J. Fac. Sci. Univ. Tokyo, Sect. IA Math. \textbf{35} (1988), no.1, 103-148.

[R] C. Robinson, \textit{Dynamical Systems (Stability, Symbolic Dynamics, and Chaos)}, CRC Press, 1999.\\

Sandra Hayes \\                   
Department of Mathematics,                                                                                                       
City College of CUNY,
New York, NY 10031: shayes@gc.cuny.edu

Axel Hundemer\\
Department of Mathematics and Statistics,
McGill University
Montreal, Quebec, CA: hundemer@math.mcgill.ca

 Evan Milliken\\
 Department of Mathematics,
 University of Florida,
 Gainesville, FL : evmilliken@ufl.edu

 Tasos Moulinos\\
 Department of Mathematics, Statistics and Computer Science
 The University of Illinois at Chicago:  tmouli2@uic.edu

\end{document}